\documentclass[a4paper]{article}
\usepackage{amsthm,amsfonts,amsmath,amssymb}
\usepackage[abbrev,nobysame]{amsrefs}
\usepackage[english]{babel}
\usepackage[final]{graphicx}
\usepackage{setspace}
\usepackage[12pt]{extsizes}
\usepackage{tikz}

\newcommand{\pattern}[4]{										% mesh pattern
	\raisebox{0.6ex}{
		\begin{tikzpicture}[scale=0.35, baseline=(current bounding box.center), #1]
		\foreach \x/\y in {#4}		\fill[gray!20] (\x,\y) rectangle +(1,1);
		\draw (0.01,0.01) grid (#2+0.99,#2+0.99);
		\foreach \x/\y in {#3}		\filldraw (\x,\y) circle (6pt);
		\end{tikzpicture}}
}

\begin{document}
%\large
%\newcommand{\per}{{\rm per}}
\newtheorem{thm}{Theorem}
\newtheorem{lemma}[thm]{Lemma}
\newtheorem{prop}[thm]{Proposition}
\newtheorem{cor}[thm]{Corollary}
\newtheorem{con}{Conjecture}
\newtheorem{rem}[thm]{Remark}
\newtheorem{quest}{Question}
\newtheorem{prob}{Problem}

\author{Sergey Avgustinovich\footnote{Sobolev Institute of Mathematics, Prospekt Akademika Koptyuga 4, Novosibirsk, 630090, Russia. 
{\bf Emails:} \{avgust,\ taa\}@math.nsc.ru.}\ , Sergey Kitaev\footnote{Department of Mathematics and Statistics, University of Strathclyde, 26 Richmond Street, Glasgow G1 1XH, United Kingdom. 
{\bf Email:} sergey.kitaev@strath.ac.uk.}\, \\ and Anna Taranenko\footnotemark[1]}
\title{On five types of crucial permutations with respect to monotone patterns}
%\date{}

\maketitle

\begin{abstract} A crucial permutation is a permutation that avoids a given set of prohibitions, but any of its extensions, in an allowable way, results in a prohibition being introduced. 

In this paper, we introduce five natural types of crucial permutations with respect to monotone patterns, notably quadrocrucial permutations that are linked most closely to Erd\H{o}s-Szekeres extremal permutations.  The way we define right-crucial and bicrucial permutations is consistent with the definition of respective permutations studied in the literature in the contexts of other prohibitions. For each of the five types, we provide its characterization in terms of Young tableaux via the RSK correspondence. Moreover, we use the characterizations to prove that the number of such permutations of length $n$ is growing when $n\to\infty$, and to enumerate minimal crucial permutations in all but one case. We also provide other enumerative results. \\
%We also offer several open problems and a conjecture.



\end{abstract}

\section{Introduction}

The celebrated Erd\H{o}s-Szekeres theorem \cite{ES1935} asserts that any permutation of length $(k-1)(\ell-1)+1$ contains an increasing subsequence of length $k$ or a decreasing subsequence of length $\ell$. It follows from the Robinson-Schensted-Knuth (RSK) correspondence and the hook-length formula \cite{Stanley} that the number of permutations of length $(k-1)(\ell-1)$ without increasing subsequences of length $k$ and decreasing subsequences of length $\ell$, $k\geq \ell$, is given by
\begin{equation}\label{Stanley-formula}
\left(\frac{((k-1)(\ell-1))!}{1^12^2\cdots (\ell-1)^{\ell-1}\ell^{\ell-1}\cdots (k-1)^{\ell-1}\ell^{\ell-2}\cdots (k+\ell-3)^1}\right)^2.
\end{equation}
This formula, in the case of $k=\ell$, has been presented in \cite{Stanley1} from 1969, and the respective permutations are called {\em Erd\H{o}s-Szekeres extremal permutations} in \cite{Vatter}. 

So, extending any permutation of length $(k-1)(\ell-1)$ that avoid increasing and decreasing subsequences of lengths $k$ and $\ell$, respectively, by an extra element to the left, or to the right, results in an occurrence of a monotone subsequence in question. However, depending on $k$ and $\ell$, there may be permutations of smaller lengths that have the same properties, namely, that avoid the monotone subsequences in question but that do not admit extensions to the right and/or to the left without introducing an unwanted monotone subsequence. Such permutations, being of significant combinatorial interest, have not received any attention in the literature, and the main goal of this paper is to fill in this gap. 

In fact, our studies in this paper belong to a much wider context. The idea of a {\em crucial object} appearing, for example, in Combinatorics on Words and Graph Theory, can be described as follows. Given a set of prohibitions $P$, we consider the set $S$ of all objects that avoid $P$, namely, the set of all objects that do not contain elements of $P$ as subobjects.  Then, a crucial object $s\in S$ is an object that cannot be extended, in a specified way, to a larger object in $S$. In other words, any attempt to enlarge $s$ in the specified way will necessarily introduce a prohibited subobject in $P$.  

In the context of our paper, the objects are permutations and the set of prohibitions is increasing subsequences of length $k$ and decreasing subsequences of length $\ell$, while extensions of permutations by an element will be defined in five ways (see Section~\ref{prelim-sec} for formal definitions): (i) just to the right (giving {\em crucial permutations}, the same as {\em right-crucial permutations}), which is equivalent to extending just to the left, or by inserting the largest element, or by inserting the smallest element; (ii) to the right and to the left (giving {\em bicrucial permutations}); (iii) to the right and by inserting the largest element (giving {\em top-right-crucial permutations}); (iv) to the right, to the left and by inserting the largest element (giving {\em tricrucial permutations}); and (v) to the right, to the left and by inserting the largest or smallest elements (giving {\em quadrocrucial permutations}). The most relevant to our paper are the studies in \cite{AKPV11,AKV12,GKKLN15,GJ22} on (bi)crucial permutations with respect to squares and arithmetic monotone patterns, although we do not use any of the results from these papers as our prohibitions are of different nature. 

Examples of natural questions to ask about crucial objects with respect to a given set of prohibitions are: Do crucial objects exist? If so, then do crucial objects exist of any size above certain size or crucial objects of maximal size exist? What is the size of a minimal crucial object?  What is the structure of crucial objects? How many crucial objects of a given size are there?  And so on. 

\begin{table}
\begin{center}
\begin{tabular}{c|c|c|c|c}
perm.\ type & character. & min length & enum.\ of min & enum. $\ell=3$ \\
\hline
\hline
crucial & Thm~\ref{cruchar} & $k+\ell-3$ & Cor.~\ref{min-crucial-count} & Cor.~\ref{k3-crucial-count} \\
\hline
top-right-cruc. & Thm~\ref{top-right-cruchar}  &  $k+\ell-3$ & Thm~\ref{thm-min-top-right-crucial} & Cor.~\ref{k3-top-right-crucial-count} \\
\hline
bicrucial & Thm~\ref{bicruchar}  & $k+2\ell-5$ & Cor.~\ref{min-bicrucial-count-cor} & only min, Cor.~\ref{min-k3-bicrucial-count-cor}\\
\hline
tricrucial & Thm~\ref{tricruchar}  & $k+2\ell-5$  & ? %Problem~\ref{enum-min-tricrucial} 
& only min, Cor.~\ref{min-k3-tricrucial-count-cor} \\
\hline
quadrocrucial & Thm~\ref{quadrocruchar}  & $k+2\ell-5$  & Thm~\ref{min-kl-quadro} & only min, Thm~\ref{min-kl-quadro}\\
\end{tabular}
\caption{A summary of our key results in the paper}\label{key-results-tab}
\end{center}
\end{table}

This paper is organized as follows. In Sections~\ref{prelim-sec} and~\ref{miscleneous-sec}  we give all necessary definitions along with a number of preliminary results and briefly review the RSK correspondence to be used in Sections~\ref{RSK-characterization-crucial}--%\ref{RSK-characterization-bicrucial}, \ref{RSK-characterization-top-right-crucial}, \ref{RSK-characterization-tricrucial} and
\ref{RSK-characterization-quadrocrucial} to characterize crucial, bicrucial, top-right-crucial, tricrucial and quadrocrucial permutations, respectively, and to provide a number of enumeration results based on the characterizations. Table~\ref{key-results-tab} summaries our main results in this paper. In Section~\ref{open-sec} we discuss directions of further research.

\section{Preliminaries}\label{prelim-sec}

A {\em pattern} is a permutation of $\{1,\ldots,k\}$ in one-line notation. An occurrence of a pattern $p=p_1\cdots p_k$ in a permutation $\sigma=\sigma_1\cdots\sigma_n$ is a subsequence $\sigma_{i_1}\cdots\sigma_{i_k}$, where $1\leq i_1<\cdots< i_k\leq n$, such that $\sigma_{i_j}<\sigma_{i_m}$ if and only if $p_j<p_m$. For example, the permutation $23514$ has two occurrences of the pattern 132, namely, the subsequences 254 and 354, while this permutation {\em avoids} (that is, has no occurrences of) the pattern 321. A pattern $p$ is {\em increasing} (resp., {\em decreasing}) if $p_1<p_2<\cdots<p_k$ (resp., $p_1>p_2>\cdots>p_k$) and we denote it $i_k$ (resp., $d_k$). Increasing and decreasing patterns are {\em monotone patterns}. Permutation patterns are an active area of research (see, for example, \cite{Kit5} and references therein). We let $S_n$ denote the set of all permutations of length $n$.

Let $\sigma = \sigma_1\cdots \sigma_n\in S_n$. Then the {\em $i$-th extension of $\sigma$ to the right} (resp., {\em left}), $1\leq i\leq n$, is the permutation  $e_i(\sigma_1)\cdots e_i(\sigma_n)i$ (resp., $ie_i(\sigma_1)\cdots e_i(\sigma_n)$), 
where 
$$e_i(x)=
\begin{cases}
x & \mbox{if }x<i, \\
x+1 & \mbox{if }x\geq i.
\end{cases}
$$
The {\em $(n+1)$-st extension} of $\sigma$ to the right (resp., left) is the permutation $\sigma(n+1)$ (resp., $(n+1)\sigma$). The {\em $i$-th extension of $\sigma$ from below} (resp., {\em above}), $1\leq i\leq n+1$, is the permutation $(\sigma_1+1)\cdots(\sigma_{i-1}+1)1(\sigma_i+1)\cdots(\sigma_n+1)$ (resp., $\sigma_1\cdots\sigma_{i-1}(n+1)\sigma_i\cdots\sigma_n$). For a sequence $\pi$ of distinct numbers, the {\em reduced form} of $\pi$, denoted red($\pi$), is the sequence obtained by replacing the $i$-th smallest element by $i$. For example, red(3826)=2413. Finally, the {\em reverse} of $\sigma$ is the permutation $r(\sigma)=\sigma_n\sigma_{n-1}\cdots\sigma_1$ and the {\em complement} of $\sigma$ is the permutation $c(\sigma) = (n+1 - \sigma_1) \cdots (n+1 - \sigma_n)$. For example, $r(2413)=3142$ and $c(25413)=41253$.

\subsection{Five types of crucial permutations}\label{5-types-of-crucial-perms-sec}

Each permutation can be viewed as a two-dimensional permutation diagram where the height of the points corresponds to the value of the respective elements.  For example, the permutation 526413 (avoiding the patterns 123 and 4321) corresponds to the permutation diagram
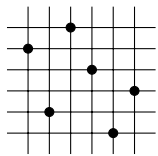
\[
\pattern{scale=0.8}{6}{1/5,2/2,3/6,4/4,5/1,6/3}{}
\]
Following a natural approach to define crucial permutations when looking at permutation diagrams, we see that, up to diagram rotation, there are five ways to specify allowed extensions of permutations to larger permutations:
\begin{itemize}
\item extending just to the right (crucial permutations in the literature correspond to extending to the right);
\item extending to the right and to the left (corresponding to bicrucial permutations in the literature);
\item extending to the right and from above; 
\item extending to the right, to the left and from above (corresponding to tricrucial permutations introduced in this paper);
\item extending to the right, to the left, from above and from below  (corresponding to quadrocrucial permutations introduced in this paper).
\end{itemize}

Next, we give formal definitions. A permutation is $(k,\ell)$-{\em right} (resp., {\em left, top, bottom}){\em-crucial} if it avoids the patterns $i_k$ and $d_{\ell}$ but any of its extensions to the right (resp., to the left, from above, from below) results in a permutation containing an occurrence of $i_k$ or $d_{\ell}$. Sometimes, we omit the word  ``right'' in a ``$(k,\ell)$-right-crucial permutation'' because ``right-crucial'' objects in the literature \cite{AKPV11,AKV12,GKKLN15} are called ``crucial''.  A permutation is $(k,\ell)$-\textit{bicrucial} if it is $(k,\ell)$-crucial and its reverse is $(\ell,k)$-crucial (equivalently, if it is $(k, \ell)$-right-crucial and $(k, \ell)$-left-crucial).  A permutation is $(k,\ell)$-\textit{top-right-crucial} if it is $(k,\ell)$-crucial and its any extension from above results in a permutation containing an occurrence of $i_k$ or $d_{\ell}$. A $(k,\ell)$-bicrucial permutation is {\em $(k,\ell)$-tricrucial} if its any extension from above results in a permutation containing an occurrence of $i_k$ or $d_{\ell}$. A $(k,\ell)$-tricrucial permutation is {\em $(k,\ell)$-quadrocrucial} if its any extension from below results in a permutation containing an occurrence of $i_k$ or $d_{\ell}$. 

Let $s^{(c)}_n(k,\ell)$ (resp., $s^{(b)}_n(k,\ell)$, $s^{(tr)}_n(k,\ell)$, $s^{(tri)}_n(k,\ell)$, $s^{(q)}_n(k,\ell)$) be the number of $(k,\ell)$-right-crucial (resp., $(k,\ell)$-bicrucial, $(k,\ell)$-top-right-crucial, $(k,\ell)$-tricrucial, $(k,\ell)$-quadrocrucial) permutations of length $n$.

For any type of permutations introduced above, a permutation of that type is {\em minimal} (resp., {\em maximal}) if there are no other permutations of the same type of smaller (resp., larger) length. Also, for any type of permutations introduced above, a permutation of that type is {\em irreducible} if there is no way to delete an element in the permutation that would result, in the reduced form, in a permutation of the same type. Finally, for any type of permutations introduced above, a permutation of that type is {\em non-extendable} if there is no way to insert a new element in the permutation (keeping the order of the other elements the same) that would result in a permutation of the same type.
%A $(k,\ell)$-(right,left,top,bottom)-crucial (resp., $(k,\ell)$-bicrucial, $(k,\ell)$-top-right, $(k,\ell)$-tricrucial, $(k,\ell)$-quadrocrucial) permutation $\sigma$ of length $n$ is {\em irreducible} if there is no way to delete an element in $\sigma$ so that the resulting permutation of length $n-1$ (in the reduced form) is $(k,\ell)$-(right,left,top,bottom)-crucial (resp., $(k,\ell)$-bicrucial, $(k,\ell)$-top-right, $(k,\ell)$-tricrucial, $(k,\ell)$-quadrocrucial). 
%
%A $(k,\ell)$-(right,left,top,bottom)-crucial (resp., $(k,\ell)$-bicrucial, $(k,\ell)$-tricrucial) permutation $\sigma$ of length $n$ is {\em non-extendable} if there is no way to insert a new element anywhere in the permutation (keeping the same relative order of the remaining elements) so that the resulting permutation of length $n+1$ is $(k,\ell)$-(right,left,top,bottom)-crucial (resp., $(k,\ell)$-bicrucial, $(k,\ell)$-tricrucial). In particular, any non-extendable $(k,\ell)$-crucial (resp., $(k,\ell)$-bicrucial, $(k,\ell)$-tricrucial) permutation is necessarily $(k,\ell)$-quadrocrucial, so the study of existence of non-extendable permutations can be reduced to the study of non-extendable $(k,\ell)$-quadrocrucial permutations.

It follows from the Erd\H{o}s-Szekeres theorem that maximal $(k,\ell)$-crucial ($(k,\ell)$-bicrucial, $(k,\ell)$-top-right-crucial, $(k,\ell)$-tricrucial, $(k,\ell)$-quadrocrucial) permutations are of length $(k-1)(\ell-1)$. However, the study of minimal lengths, and indeed, $(k,\ell)$-(bi,top-right-,tri,quadro)crucial permutations of lengths between minimal and maximal, is an interesting and challenging research direction.

\subsection{The RSK correspondence}

The {\em Robinson-Schensted-Knuth correspondence}, also referred to as the {\em RSK correspondence} or the {\em RSK algorithm}, is a bijection between the set of all permutations of length $n$ and pairs $(P,Q)$ of {\em standard Young tableaux} of the same shape. Even though the RSK correspondence is a well known object in algebraic combinatorics, we briefly review it here as several of our key results in this paper rely on it.

Let $\lambda=(\lambda_1,\lambda_2,\ldots)$ be a partition of $n\geq 0$, denoted $\lambda \vdash n$. Hence, $\lambda_1\geq\lambda_2\geq\cdots\geq0$ and $\sum\lambda_i=n$. The {\em Young diagram} of a partition $\lambda$ is a left-justified array of squares with $\lambda_i$ squares in the $i$-th row.  For example, the Young diagram of (4,3,1) is given by 

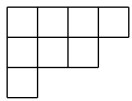
\begin{center}
\begin{tikzpicture}[scale=0.4]

\draw [line width=0.5](0,3)--(4,3);
\draw [line width=0.5](0,2)--(4,2);
\draw [line width=0.5](0,1)--(3,1);
\draw [line width=0.5](0,0)--(1,0);

\draw [line width=0.5](0,0)--(0,3);
\draw [line width=0.5](1,0)--(1,3);
\draw [line width=0.5](2,1)--(2,3);
\draw [line width=0.5](3,1)--(3,3);
\draw [line width=0.5](4,2)--(4,3);

\end{tikzpicture}
\end{center}

A standard Young tableau (SYT) of shape $\lambda \vdash n$ is obtained by placing the integers $1,\ldots,n$ (each appearing once) into the squares of the diagram of $\lambda$ (with one integer in each square) such that every row and column is increasing. For example, an SYT of shape (4,2,1) is given by 

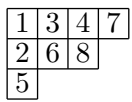
\begin{center}
\begin{tikzpicture}[scale=0.4]

\draw [line width=0.5](0,3)--(4,3);
\draw [line width=0.5](0,2)--(4,2);
\draw [line width=0.5](0,1)--(3,1);
\draw [line width=0.5](0,0)--(1,0);

\draw [line width=0.5](0,0)--(0,3);
\draw [line width=0.5](1,0)--(1,3);
\draw [line width=0.5](2,1)--(2,3);
\draw [line width=0.5](3,1)--(3,3);
\draw [line width=0.5](4,2)--(4,3);

\node [below] at (0.5,3.2){{\small 1}};
\node [below] at (1.5,3.2){{\small 3}};
\node [below] at (2.5,3.2){{\small 4}};
\node [below] at (3.5,3.2){{\small 7}};

\node [below] at (0.5,2.2){{\small 2}};
\node [below] at (1.5,2.2){{\small 6}};
\node [below] at (2.5,2.2){{\small 8}};

\node [below] at (0.5,1.2){{\small 5}};

\end{tikzpicture}
\end{center}

The basic operation of the RSK correspondence is {\em row insertion}, that is, inserting an integer $i$ into a tableau $T$ with distinct entries and with increasing rows and columns. Hence, $T$ satisfies the conditions of an SYT except that its entries can be any distinct integers, not necessarily $1,\ldots,n$. The process of row inserting $i$ into $T$ produces another tableau, denoted $T\leftarrow i$, with increasing rows and columns. If $S$ is the set of entries of $T$, then $S\cup\{i\}$ is the set of entries of $T \leftarrow i$. Now, $T \leftarrow i$ is defined recursively as follows. 
\begin{itemize}
\item If the first row of $T$ is empty or the largest entry of the first row of $T$ is less than $i$, then insert $i$ at the end of the first row.
\item Otherwise, $i$ replaces (or {\em bumps}) the smallest element $j>i$ in the first row of $T$. We then insert $j$ into the second row of $T$ by the same procedure.
\end{itemize}

Let $\sigma=\sigma_1\cdots\sigma_n\in S_n$, and let $\emptyset$ denote the empty tableau. Define
$$P_i=P_i(\sigma)=(\cdots((\emptyset \leftarrow \sigma_1)\leftarrow \sigma_2)\leftarrow\cdots\leftarrow \sigma_i$$
and set $P:=P(\sigma)=P_n(\sigma)$. $P$ is called the {\em insertion tableau}. Define $Q_0=\emptyset$, and once $Q_{i-1}$ is defined let $Q_i=Q_i(\sigma)$ be obtained from $Q_{i-1}$ by inserting $i$ (without changing the position of any of the entries of $Q_{i-1}$) so that $Q_i$ and $P_i$ have the same shape. Set $Q:=Q(\sigma)=Q_n(\sigma)$ called the {\em recording tableau}, and finally define the output of the RSK correspondence applied to $\sigma$ to be the pair $(P,Q)$. We refer to \cite{Stanley} for an example of the RSK correspondence applied to a permutation of length 5.

\begin{rem}\label{longest-remark} \normalfont It is known that the number of columns in $P$ (and $Q$) is the length of a longest increasing subsequence in $\sigma$, and the number of rows in the tableaux is the length of a longest decreasing subsequence. \end{rem}

Below, we let $p_{i,j}$ and $q_{i,j}$ denote the entry in row $i$ (from above) and column $j$ (from left) in $P$ and $Q$, respectively. In what follows, we will also need the {\em evacuation map}  introduced in~\cite{Schutz}. Here we provide one of its equivalent definitions. 

First, we introduce jeu-de-taquin slides of a tableau shape. Let $a < b$ and $\bullet$ represents an empty square. Then jeu-de-taquin slides are

\begin{center}
\begin{tabular}{ccc}
\begin{tikzpicture}[scale=0.4]

\draw [line width=0.5](1,2)--(2,2);
\draw [line width=0.5](0,1)--(2,1);
\draw [line width=0.5](0,0)--(2,0);

\draw [line width=0.5](0,0)--(0,1);
\draw [line width=0.5](1,0)--(1,2);
\draw [line width=0.5](2,0)--(2,2);

\node [below] at (0.5,2.2){{\small $\bullet$}};
\node [below] at (1.5,2.1){{\small $a$}};

\node [below] at (0.5,1.2){{\small $b$}};
\node [below] at (1.5,1.1){{\small $c$}};

\node [below] at (3,1.5){$\rightarrow$};

\draw [line width=0.5](4,2)--(5,2);
\draw [line width=0.5](4,1)--(6,1);
\draw [line width=0.5](4,0)--(6,0);

\draw [line width=0.5](4,0)--(4,2);
\draw [line width=0.5](5,0)--(5,2);
\draw [line width=0.5](6,0)--(6,1);

\node [below] at (4.5,2.1){{\small $a$}};
\node [below] at (5.5,2.2){{\small $\bullet$}};

\node [below] at (4.5,1.2){{\small $b$}};
\node [below] at (5.5,1.1){{\small $c$}};

\end{tikzpicture}

& \ \ \ &

\begin{tikzpicture}[scale=0.4]

\draw [line width=0.5](1,2)--(2,2);
\draw [line width=0.5](0,1)--(2,1);
\draw [line width=0.5](0,0)--(2,0);

\draw [line width=0.5](0,0)--(0,1);
\draw [line width=0.5](1,0)--(1,2);
\draw [line width=0.5](2,0)--(2,2);

\node [below] at (0.5,2.2){{\small $\bullet$}};
\node [below] at (1.5,2.2){{\small $b$}};

\node [below] at (0.5,1.1){{\small $a$}};
\node [below] at (1.5,1.1){{\small $c$}};

\node [below] at (3,1.5){$\rightarrow$};

\draw [line width=0.5](4,2)--(6,2);
\draw [line width=0.5](4,1)--(6,1);
\draw [line width=0.5](5,0)--(6,0);

\draw [line width=0.5](4,1)--(4,2);
\draw [line width=0.5](5,0)--(5,2);
\draw [line width=0.5](6,0)--(6,2);

\node [below] at (4.5,2.1){{\small $a$}};
\node [below] at (5.5,2.2){{\small $b$}};

\node [below] at (4.5,1.2){{\small $\bullet$}};
\node [below] at (5.5,1.1){{\small $c$}};
\end{tikzpicture} 
\end{tabular}
\end{center}
Here, the squares $b$ and $c$ may or may not be present.

Let $T$ be a SYT with $n$ squares. Place $T$ in a tight rectangle and replace every entry $j$ of $T$ by $n +1 -j$. Then rotate the rectangle $180^{\circ}$ and perform jeu-de-taquin slides on the result until a standard Young tableau is obtained. The constructed  SYT is the \textit{evacuation tableau} $\epsilon (T)$.

\section{Miscellaneous results}\label{miscleneous-sec}

\begin{lemma}\label{lemma-1}
Let $\sigma=\sigma_1\cdots\sigma_n$ be a $(k,\ell)$-crucial permutation. Then, $\sigma$ has an occurrence of $i_{k-1}$ and an occurrence of $d_{\ell-1}$ both ending at $\sigma_n$ (and having no other elements in common). 
\end{lemma}

\begin{proof}
Consider extending $\sigma$ by the element $\sigma_{n+1}:=\sigma_{n}+1$, which cannot introduce an occurrence of the pattern $d_{\ell}$ as otherwise, such an occurrence is given by $\sigma_{j_1}>\cdots>\sigma_{j_{\ell-1}}>\sigma_{n+1}$ for some $j_1<\cdots<j_{\ell-1}<n$ and $\sigma_{j_1}>\cdots>\sigma_{j_{\ell-1}}>\sigma_{n}$ would have been an occurrence of $d_{\ell}$ in $\sigma$ (since $\sigma_{n+1}-\sigma_n=1$), which is impossible. Hence, such an extension must introduce an occurrence of the pattern $i_k$ as $\sigma$ is $(k,\ell)$-crucial. Since $\sigma_{n+1}-\sigma_n=1$, such an occurrence must involve $\sigma_{n}$ or else there would be an occurrence of $i_k$ in $\sigma$ ending with $\sigma_n$. That shows that $\sigma$ has an occurrence of $i_{k-1}$  ending with $\sigma_n$. 

On the other hand, consider extending $\sigma$ by the element $\sigma_{n+1}:=\sigma_n$, which cannot introduce an occurrence of the pattern $i_k$ as otherwise, such an occurrence is given by $\sigma_{j_1}<\cdots<\sigma_{j_{k-1}}<\sigma_{n+1}$ for some $j_1<\cdots<j_{k-1}<n$ and $\sigma_{j_1}<\cdots<\sigma_{j_{k-1}}<\sigma_{n}$ would have been an occurrence of $i_k$ in $\sigma$ (since $\sigma_{n+1}$ is one less than the element next to it in the obtained permutation of length $n+1$), which is impossible. Hence, such an extension must introduce an occurrence of the pattern $d_{\ell}$ as $\sigma$ is $(k,\ell)$-crucial. Since $\sigma_{n}-\sigma_{n+1}=1$, such an occurrence must involve $\sigma_{n}$ or else there would be an occurrence of $d_{\ell}$ in $\sigma$ ending with $\sigma_n$. That shows that $\sigma$ has an occurrence of $d_{\ell-1}$  ending with $\sigma_n$. 

Clearly, the only common element in the occurrences of the patterns $i_{k-1}$ and $d_{\ell-1}$ in question can be the final element $\sigma_n$.
\end{proof}

The proof of the next lemma is analogous to that of Lemma~\ref{lemma-1} and hence is omitted. 

\begin{lemma}\label{lemma-2-like-lemma-1}
Let $\sigma=\sigma_1\cdots\sigma_n$ be a $(k,\ell)$-left-crucial permutation. Then, $\sigma$ has an occurrence of $i_{k-1}$ and an occurrence of $d_{\ell-1}$ both beginning at $\sigma_1$ (and having no other elements in common). 
\end{lemma}

\begin{thm}\label{thm-min-crucial}
Any minimal $(k,\ell)$-crucial permutation $\sigma$ is of length $n = k+\ell -3$ and it has the following form: 
\begin{itemize}
\item the final element is $\sigma_n = k-1$; 
\item the first $n -1 = k-\ell-4$ positions of $\sigma$ are occupied by the interleaved occurrences of the pattern $i_{k-2}$ formed by the elements in $\{1, \ldots, k-2\}$ and the pattern $d_{\ell-2}$ formed by the elements in $\{k,\ldots,k+\ell-3\}$. In particular, any minimal $(k,\ell)$-crucial permutation begins with either the smallest or the largest element.
\end{itemize}
\end{thm}

\begin{proof} Since an increasing and a decreasing subsequences in any permutation can have at most one element in common, we see that, by Lemma~\ref{lemma-1}, the length of a minimal $(k,\ell)$-crucial permutation is at least $k+\ell -3$. Moreover, by Lemma~\ref{lemma-1}, we see that the length of any minimal $(k,\ell)$-crucial permutation is $k+\ell-3$ and its structure is as in the statement of the theorem. \end{proof}

\begin{rem}\label{remark-similar-lemma-2} \normalfont It follows immediately from Lemma~\ref{lemma-1}  that every $(k,\ell)$-crucial permutation can be obtained by inserting elements to a minimal $(k,\ell)$-crucial permutation (without changing the relative order of already existing elements). Indeed, given a $(k,\ell)$-crucial permutation we can locate $k+\ell-3$ ``essential elements'' in it discussed in Lemma~\ref{lemma-1}, then remove the remaining elements to obtain a minimal $(k,\ell)$-crucial permutation. Reversing the removing steps, we will obtain the desired result. \end{rem}

The proof of the following lemma, that is closely related to Remark~\ref{remark-similar-lemma-2}, is straightforward, and hence is omitted. 

\begin{lemma}\label{lemma-2}
Let $\sigma=\sigma_1\cdots\sigma_n$ be a $(k,\ell)$-crucial permutation. Let $I$ be a collection of $k+\ell-3$ elements in $\sigma$ forming occurrences of $i_{k-1}$ and $d_{\ell-1}$ both ending at $\sigma_n$ (that exist by Lemma~\ref{lemma-1}). Then removing any element $\sigma_i\not\in I$ and taking the reduced form results in a $(k,\ell)$-crucial permutation in $S_{n-1}$.
\end{lemma}

Theorem~\ref{thm-min-crucial}, Lemma~\ref{lemma-2} and formula (\ref{Stanley-formula}) give the following result.

\begin{cor}
There exist $(k,\ell)$-crucial permutations in $S_n$ for each $n$, $k+\ell-3 \leq n \leq (k-1)(\ell-1)$.
\end{cor}

The following corollary is immediate from Theorem~\ref{thm-min-crucial}, and we invite the reader to compare this result with a much more sophisticated formula (\ref{Stanley-formula}) for the number of maximal $(k,\ell)$-crucial permutations.

\begin{cor}\label{min-crucial-count}
The number of minimal  $(k,\ell)$-crucial permutations  is ${k+\ell-4 \choose k-2}$.
\end{cor}

Another comparison that one can make is between minimal and non-extendable $(k,\ell)$-crucial permutations: while the former always begin with the smallest or the largest elements, the later never begin with these elements. Indeed, suppose that a non-extendable $(k,\ell)$-crucial permutation $\sigma\in S_n$ begins with 1 (resp., $n$) and $k,\ell\geq 3$. Then, extending $\sigma$ to the left by $2$ (resp., $n$) we obtain a permutation in $S_{n+1}$ that is clearly $(k,\ell)$-crucial, which  contradicts to $\sigma$ being non-extendable.

Yet another corollary of Theorem~\ref{thm-min-crucial} is the following result. 

\begin{thm}\label{thm-min-bicrucial-structure}
Suppose that $k\geq\ell$. Any minimal $(k,\ell)$-bicrucial permutation $\sigma=\sigma_1\cdots\sigma_n$ is of length $n = k+2\ell -5$ and it has the following form: 
\begin{itemize}
\item $\sigma_1=\ell-1$;
\item $\sigma_n = k+\ell-3$; 
\item the elements in $\{1,2,\ldots,\ell-1\}$ and those in $\{k+\ell-3,\ldots,k+2\ell-5\}$ are in decreasing order, while the elements in $\{\ell-1,\ell,\ldots,k+\ell-3\}$ are in increasing order.
\end{itemize}
\end{thm}

\begin{proof} Since $\sigma$ is $(k,\ell)$-bicrucial, it is also $(k,\ell)$-crucial, and by Lemma~\ref{lemma-1}, $\sigma$ has an occurrence of $i_{k-1}$, say formed by the elements in a set $A$, and an occurrence of $d_{\ell-1}$, say formed by the elements in a set $B$, both ending at $\sigma_n$ (hence $A\cap B=\{\sigma_n\}$). 

Moreover, since $\sigma$ is also $(k,\ell)$-left-crucial, by Lemma~\ref{lemma-2-like-lemma-1}, $\sigma$ has an occurrence of $i_{k-1}$, say formed by the elements in a set $C$, and an occurrence of $d_{\ell-1}$, say formed by the elements in a set $D$, both beginning at $\sigma_1$ (hence $C\cap D=\{\sigma_1\}$). 

Since an increasing and a decreasing sequences can share at most one element, any minimal $(k,\ell)$-bicrucial permutation must have $A=C$ (which minimizes the length of a $(k,\ell)$-bicrucial permutation; here we use the fact that $k\geq\ell$ and hence did not consider the other possibility of minimizing the length, namely, $B=D$). From this observation, we recover immediately that $D=\{1,2,\ldots,\ell-1\}$, $A=C=\{\ell-1,\ell,\ldots,k+\ell-3\}$, $B=\{k+\ell-3,\ldots,k+2\ell-5\}$,  $\sigma_1=\ell-1$, $\sigma_n = k+\ell-3$ and $n=k+2\ell -5$.

 \end{proof}

\begin{thm}\label{thm-min-top-right-crucial}
There is exactly one minimal $(k,\ell)$-top-right-crucial permutation that is of length $k+\ell-3$ and is of the form $$12\cdots (k-2)(k+\ell-3)(k+\ell-4)\cdots (k-1).$$
\end{thm}

\begin{proof} Any $(k,\ell)$-top-right-crucial permutation $\sigma$ is $(k,\ell)$-crucial and hence Theorem~\ref{thm-min-crucial}  can be applied to learn about the relative positions of elements in $A=\{1, \ldots, k-2\}$, those in $B=\{k,\ldots,k+\ell-3\}$, and the element $k-1$. Now, extending $\sigma$ from above must result in an occurrence of $i_k$ or $d_{\ell}$. Inserting $k+\ell-2$ (the new maximum element) anywhere to the left of $k+\ell-3$ results in an occurrence of $d_{\ell}$ involving all the elements in $B$ and the element $k-1$. On the other hand, inserting $k+\ell-2$ to the right of $k+\ell-3$ cannot result in an occurrence of $d_{\ell}$ and hence must result in an occurrence of $i_k$ that can involve only one element in $B\cup\{k-1\}$. Considering inserting $k+\ell-2$ immediately to the right of $k+\ell-3$, we see that all the elements in $A$ must be placed (in increasing order) to the left of $k+\ell-3$, which gives the required structure. \end{proof}

To characterize $(k,\ell)$-left-, $(k,\ell)$-top- and $(k,\ell)$-bottom-crucial permutations in Section~\ref{RSK-characterization-crucial}, we need the following result.

\begin{lemma}\label{crucial-types}
Let $\sigma \in S_n$. The following statements are equivalent: 
\begin{itemize}
\item $\sigma$ is $(k,\ell)$-right-crucial;
\item the reverse permutation $r(\sigma)$  is $(\ell, k)$-left-crucial;
\item the inverse permutation $\sigma^{-1}$ is $(k,\ell)$-top-crucial;
\item the complement of the inverse permutation $c(\sigma^{-1})$ is $(\ell, k)$-bottom-crucial.
\end{itemize}
\end{lemma}

\begin{proof}
Let $\mathcal{D}$ be the two-dimensional diagram corresponding to a $(k,\ell)$-right-crucial permutation $\sigma$. By definition, adding a column to $\mathcal{D}$ from the right results in an occurrence of the pattern $i_k$ or $d_{\ell}$.

The reverse permutation $r(\sigma)$ has the diagram $r(\mathcal{D})$ obtained from $\mathcal{D}$ by reversing the order of its columns. Any occurrence of the pattern $i_m$ in $\sigma$ corresponds to an occurrence of the pattern $d_{m}$ in $r(\sigma)$, and vice versa. Hence, non-extendability of $\sigma$ to the right with respect to $i_k$ and $d_{\ell}$ is equivalent to non-extendability of $r(\sigma)$ to the left with respect to $i_{\ell}$ and $d_{k}$.

The inverse permutation $\sigma^{-1}$ has the diagram $\mathcal{D}'$ obtained from $\mathcal{D}$ by transposition with respect to the secondary diagonal. Every occurrence of the pattern $i_m$ (resp., $d_m$) in $\sigma$ corresponds to an occurrence of the pattern $i_{m}$ (resp., $d_m$) in $\sigma^{-1}$. Hence, non-extendability of $\sigma$ to the right with respect to $i_k$ and $d_{\ell}$ is equivalent to non-extendability of $\sigma^{-1}$ from above with respect to $i_k$ and $d_{\ell}$.

The complement permutation $c(\sigma)$ has the diagram $c(\mathcal{D})$ obtained from $\mathcal{D}$ by reversing the order of its rows. Any occurrence of the pattern $i_m$ in $\sigma$ corresponds to an occurrence of the pattern $d_{m}$ in $c(\sigma)$, and vice versa. Hence, non-extendability of $\sigma$ to the right with respect to $i_k$ and $d_{\ell}$ is equivalent to non-extendability of $c(\sigma^{-1})$ from below with respect to $i_{\ell}$ and $d_{k}$.
\end{proof}

\section{Characterization of $(k,\ell)$-crucial permutations via RSK correspondence}\label{RSK-characterization-crucial}

\begin{thm} \label{cruchar}
Let $\sigma\in S_n$ and $(P,Q)$ be the pair of SYT corresponding to $\sigma$. Then, $\sigma$ is $(k,\ell)$-right-crucial if and only if the following holds:
\begin{itemize}
\item the number of columns in $P$ (and $Q$) is $k-1$;
\item the number of rows in $P$ (and $Q$) is $\ell-1$;
\item $P$  contains an increasing sequence of elements $p_{1, k-1}, p_{2,j_2}, \ldots, p_{\ell-1, j_{\ell-1}}$ for some $j_2, \ldots, j_{\ell-1}$.
\end{itemize}
\end{thm}

\begin{proof}
``$\Leftarrow$'' Assume that $(P,Q)$ satisfies the requirements. By Remark~\ref{longest-remark}, the longest increasing (resp., decreasing) subsequence in $\sigma$ is of length $k-1$ (resp., $\ell-1$). 

Let us extend $\sigma$ to the right by an element $t$. Suppose that $\sigma'$ is the resulting permutation corresponding to the pair of SYT $(P',Q')$. If $t>p_{1, k-1}$ then the extension of $\sigma$ by $t$ results in a new column added to $P$, and hence in an occurrence of $i_k$ in $\sigma'$. 

Assume that $t \leq p_{1,k-1}$ so that the element $p_{1,k-1}$ in $\sigma$ and $P$ becomes the element $(p_{1,k-1}+1)$ in $\sigma'$ and $P'$. By the row insertion operation of the RSK correspondence, we have that $t$ is inserted into the first row of $P$ and an element $p_{1,s_1}$ in $P$, $p_{1, s_1} < p_{1, k-1}$, will be moved to the second row of $P$. Since $p_{1,s_1}$ is smaller than the maximum element in the second row of $P$ (because $p_{1,s_1}$  is smaller than $p_{2,j_2}$), the insertion algorithm replaces an element $p_{2, s_2}$ in the second row of  $P$, $p_{2, s_2} < p_{2, j_2}$, by $p_{1,s_1}$. Continuing in this way, we see that an element $p_{\ell-2, s_{\ell-2}}$ replaces an element $p_{\ell-1, s_{\ell-1}}$, where $p_{\ell-1, s_{\ell-1}} < p_{\ell-1, j_{\ell-1}}$, and the element  $p_{\ell-1, s_{\ell-1}}$ goes to the $\ell$-th row in $P'$. Consequently, by Remark~\ref{longest-remark}, $\sigma'$ has an occurrence of $d_{\ell}$ showing that $\sigma$ is $(k,\ell)$-right-crucial. 

``$\Rightarrow$'' By the definition of a $(k,\ell)$-right-crucial permutation and Remark~\ref{longest-remark}, the requirements on the number of rows and columns are satisfied.

Assume that there is no increasing sequence $p_{1, k-1}, p_{2,j_2}, \ldots, p_{\ell-1, j_{\ell-1}}$ in $P$. Let us consider the longest increasing sequence $p_{1, k-1}, p_{2,j_2}, \ldots, p_{r, j_{r}}$ such that it is lexicographically smallest among all such longest increasing sequences starting at $p_{1, k-1}$.  By our assumption, $1\leq r < \ell-1$ and all elements in the $(r+1)$-th row of $P$ are smaller than $p_{r,j_r}$, where $j_1:=k-1$. In particular, since $P$ is an SYT, we have that the $(r+1)$-th row of $P$ is shorter than the $r$-th row.

Let us extend $\sigma$ to the right by the element $p_{1, k-1}$ so that the elements $p_{1, k-1}, p_{2,j_2}, \ldots, p_{r, j_{r}}$ in $\sigma$ and $P$ become, respectively, the elements $(p_{1, k-1}+1), (p_{2,j_2}+1), \ldots, (p_{r, j_{r}}+1)$ in $\sigma'$ and $P'$.   Note that the row insertion operation will move the elements $(p_{1, k-1}+1), (p_{2,j_2}+1), \ldots, (p_{r, j_{r}}+1)$ and add one  new element at the end of the  $(r+1)$-th row:
$$p'_{1, k-1} = p_{1, k-1};~ p'_{i, j_i} = p_{i-1, j_{i-1}} \mbox{ for }  i \in  \{2, \ldots, r \}; ~ p'_{r+1, s_{r} +1} = p_{r,j_r},$$
where $s_r$ is the length of the $r$-th row in $P$. We see that the lengths of the longest increasing and decreasing subsequences  in $\sigma$ and $\sigma'$ are the same contradicting the fact that $\sigma$ is a $(k,\ell)$-right-crucial permutation. Hence, $P$  contains an increasing sequence of elements $p_{1, k-1}, p_{2,j_2}, \ldots, p_{\ell-1, j_{\ell-1}}$ for some $j_2, \ldots, j_{\ell-1}$.
\end{proof}

To illustrate applications of Theorem~\ref{cruchar}, note that the SYT $P$ given by 

\begin{center}
\begin{tikzpicture}[scale=0.4]

\draw [line width=0.5](0,2)--(3,2);
\draw [line width=0.5](0,1)--(3,1);
\draw [line width=0.5](0,0)--(2,0);

\draw [line width=0.5](0,0)--(0,2);
\draw [line width=0.5](1,0)--(1,2);
\draw [line width=0.5](2,0)--(2,2);
\draw [line width=0.5](3,1)--(3,2);

\node [below] at (0.5,2.2){{\small 1}};
\node [below] at (1.5,2.2){{\small 2}};
\node [below] at (2.5,2.2){{\small 3}};

\node [below] at (0.5,1.2){{\small 4}};
\node [below] at (1.5,1.2){{\small 5}};

\end{tikzpicture}
\end{center}

\noindent
satisfies all conditions in the theorem when $k=4$ and $\ell=3$, and hence, since there are five choices for $Q$ in this case, we see that there are five $(4,3)$-crucial permutations with such $P$: 14523, 14253, 45123, 41253, 41523 (the remaining $(4,3)$-crucial permutations of length 5 are 13254, 13524, 21354, 21534, 23154, 23514, 25134, 31254, 31524, 35124).

As an immediate corollary to Theorem~\ref{crucial-growth-thm} and Corollary~\ref{min-crucial-count}, we have the following result.

\begin{cor}\label{k3-crucial-count} The number of $(k,3)$-right-crucial permutations of length $n=k$ is $k-1$, and for $n\geq k+1$ and $k\geq 3$, it is given by 
\begin{equation}\label{eqn-snk3}
s^{(c)}_n(k,3)=\frac{(2k-n)(2k-n-1)}{(n-k)(n-k+1)}{n-1 \choose k}{n\choose k}.
\end{equation}
Moreover, for $2\leq i\leq k-1$, the number of $(k,3)$-right-crucial permutations of $i$-th smallest length $k+i-1$ is given by
\begin{equation}\label{eqn-snk3i}
\frac{(k-i)(k-i+1)}{i(i-1)}{k+i-2 \choose k}{k+i-1\choose k}
\end{equation}
and this number is $0$ for $i\geq k$. By symmetry, replacing $k$ by $\ell$ in the formulas, we obtain the results for $(3,\ell)$-crucial permutations for $\ell\geq 3$.  
\end{cor}

\begin{proof} The fact that there are $k-1$ $(k,3)$-right-crucial permutations of length $k$ is obtained from Corollary~\ref{min-crucial-count} by letting $\ell=3$.  

Let $\sigma$ be a $(k,3)$-right-crucial permutation of length $n\geq k+1$ and the pair of SYT $(P,Q)$ corresponds to $\sigma$. Then, by Theorem~\ref{crucial-growth-thm}, the first row of $P$ is of length $k-1$, the second row of $P$ is of length $n-k+1$ and also $p_{2,n-k+1}=n$. We can now use the hook length formula to count the number of SYT with two rows obtained by removing the element  $p_{2,n-k+1}$ (any such SYT has two rows because  $n\geq k+1$; each such SYT gives a unique choice for $P$): 
\begin{equation}\label{SYT-P-formula-k3}
\frac{(2k-n)(n-1)!}{k!(n-k)!}=\frac{2k-n}{n-k}{n-1\choose k}.
\end{equation}
Also, by the hook length formula, the number of choices for $Q$ is given by 
$$\frac{(2k-n-1)n!}{k!(n-k+1)!}=\frac{2k-n-1}{n-k+1}{n\choose k}.$$
Multiplying the formulas above, we obtain (\ref{eqn-snk3}). 

To prove (\ref{eqn-snk3i}), first note that by Theorem~\ref{thm-min-crucial}, the minimal $(k,3)$-crucial permutations are of length $k$, and  by the Erd\H{o}s-Szekeres theorem \cite{ES1935} the maximal $(k,3)$-crucial permutations are of length $(k-1)(3-1)$ giving us the bounds for $i$ as the $i$-th smallest $(k,3)$-crucial permutations are of length $k+i-1$.  We now obtain (\ref{eqn-snk3i}) by inserting $n=k+i-1$ in (\ref{eqn-snk3}). \end{proof}

\begin{rem}\label{next-min-crucial-count} \normalfont Corollary~\ref{k3-crucial-count}  allows us to derive the following results. The number of $(k,3)$-right-crucial permutations of length $k+1$ (next smallest length, the case of $i=2$), is given by
$$\frac{(k-2)(k^2-1)}{2}.$$ The corresponding sequence $4,15, 36, 70, 120, 189, 280,\ldots$ is the sequence A077414 in \cite{oeis} with several combinatorial interpretations. 

The number of $(k,3)$-right-crucial permutations of length $k+2$ (the case of $i=3$), is given by
$$\frac{(k-3)(k^2-4)(k+1)^2}{12},$$
and the respective sequence begins with $0, 25, 126, 392, 960, 2025, 3850, 6776, \ldots$. \end{rem}

\begin{thm}\label{crucial-growth-thm} For $k+\ell-3\leq n < (k-1)(\ell-1)$, we have the following monotone property for the number of crucial permutations:
\begin{equation}\label{monot-cru}s^{(c)}_n(k,\ell)< s^{(c)}_{n+1}(k,\ell).\end{equation} 
\end{thm}

\begin{proof} Suppose that $\sigma\in S_n$ is a $(k,\ell)$-crucial permutation and $(P,Q)$ is the pair of SYT corresponding to $\sigma$. Then $P$ and $Q$ satisfy the conditions of Theorem~\ref{cruchar}. Moreover, since $n<(k-1)(\ell-1)$,  $P$ has an element $p_{i,j}$ such that $2\leq i\leq \ell-1$ , $1\leq j<k-1$ and there is no element to the right of $p_{i,j}$ (that is, $p_{i,j+1}$ does not exist). Out of all such $p_{i,j}$ we choose one with the minimum $i$. Let the SYT $P'$ (resp., $Q'$) be obtained from $P$ (resp., $Q$) by introducing the new element $p_{i,j+1}=n+1$ (resp., $q_{i,j+1}=n+1$), which we think of as filling in the corner next to $p_{i,j}$. Let the pair $(P',Q')$ correspond to a permutation $\sigma'$ of length $n+1$ via the RSK correspondence. Note that $P'$ satisfies the conditions of Theorem~\ref{cruchar}, so $\sigma'$ is a $(k,\ell)$-crucial permutation. Moreover, the operation of obtaining $\sigma'$ from $\sigma$ is clearly injective, and hence $s^{(c)}_n(k,\ell)\leq s^{(c)}_{n+1}(k,\ell)$. 

To prove that the inequality is actually strict, we consider two cases:
\begin{itemize}
\item[(i)] Let $k+\ell-3\leq n\leq (k-1)(\ell-2)$. In this case, there exist Young diagrams with $n+1$ squares, $k-1$ columns and $\ell-1$ rows with exactly one square in the last row. Let $P$ and $Q$ be SYT of the same shape with $n$ entries, $k-1$ columns and $\ell-2$ rows, and $P$ satisfies the last condition in Theorem~\ref{cruchar} (with the non-existing $p_{\ell-1,j_{\ell-1}}$ ignored), so that the pair  $(P,Q)$ corresponds to a permutation counted by $s^{(c)}_n(k,\ell-1)$.We can extend $P$ and $Q$ to SYT $P'$ and $Q'$, respectively, by adjoining the element $n+1$ in the row $\ell-1$, which is an injective operation. Clearly, $P'$  satisfies the last condition in Theorem~\ref{cruchar} and hence the permutation corresponding to the pair $(P',Q')$ is counted by $s^{(c)}_{n+1}(k,\ell)$. Since the element $n+1$ does not have another element to the left of it, the $(k,\ell)$-crucial permutations of length $n+1$ obtained in this way are distinct from the ones considered above. Hence, in this case \begin{equation}\label{casei-impr} s^{(c)}_n(k,\ell)<s^{(c)}_n(k,\ell)+s^{(c)}_n(k,\ell-1)\leq s^{(c)}_{n+1}(k,\ell).\end{equation}
\item[(ii)] Let $(k-1)(\ell-2)<n<(k-1)(\ell-1)$. In this case, we can consider Young diagrams having $k-1$ squares in each of the first $\ell-2$ rows and $d:=n-(k-1)(\ell-2)\geq 1$ squares in the last row. We define a SYT $P$ by filling in such a Young diagram with elements as follows:
\begin{itemize}
\item[(a)] the entries $p_{i,j}$ for $1\leq i\leq\ell-1$ and $1\leq j\leq d$ are in $\{1,\ldots, (\ell-1)d\}$ and they form a SYT. For example, reading from left to right, we can let row 1 be $1, 2,\ldots, d$, row 2 be $d+1,d+2,\ldots,2d$, etc.  
\item[(b)] the remaining entries $p_{i,j}$ for $1\leq i\leq\ell-2$ and $d+1\leq j\leq k-1$ also form a SYT, which will automatically satisfy the last condition in Theorem~\ref{cruchar} because of the last column (with the non-existing $p_{\ell-1,j_{\ell-1}}$ ignored). An example of an arrangement here can be, again, placing the elements consecutively in increasing order from left to right from top to bottom with the element $(\ell-1)d+1$ being in the North-West corner. 
\end{itemize} 
We note that because row $\ell-1$ has relatively ``small'' elements, the constructed by us $P$ does not satisfy the last condition in Theorem~\ref{cruchar}. However, adjoining to $P$ the element $p_{\ell-1,d+1}=n+1$ gives a SYT $P'$ satisfying all conditions in Theorem~\ref{cruchar}. Hence, choosing any SYT $Q'$ with $n+1$ elements of the shape of $P'$ we obtain a pair of SYT $(P',Q')$ corresponding to a $(k,\ell)$-crucial permutation of length $n+1$ that was not considered above (including case (i)), because the element $n+1$ has another element to the left of it, and removing $n+1$ from $P'$ does not give a SYT satisfying the last condition in Theorem~\ref{cruchar}. Hence, in this case we also have (\ref{monot-cru}).
\end{itemize}
\noindent 
{\bf An alternative proof of strictness of $s^{(c)}_n(k,\ell)\leq s^{(c)}_{n+1}(k,\ell)$.} Note that for $k+\ell-3<n<(k-1)(\ell-1)-2$ we can always construct a SYT $P$ with at least two corners satisfying the conditions in Theorem~\ref{cruchar} (e.g. by placing the elements consecutively in increasing order from left to right from top to bottom in an appropriate Young diagram). That means that we can extend $P$ in more than one way by placing $n+1$ in another corner (not just in the top one as is done above), and such an extension is clearly injective. In the remaining cases we have exactly one corner. If $n\in\{(k-1)(\ell-1)-2,(k-1)(\ell-1)-1\}$ then we can apply the construction in case (ii) above to get an extra $(k,\ell)$-crucial permutation of length $n+1$. Finally, if $n=k+\ell-3$ then note that $p_{\ell-1,1}=n$. We can then extend $P$ to $P'$ by letting the corner $p'_{2,2}=n$, the bottommost element $p'_{\ell-1,1}=n+1$, and keeping all other elements of $P$ the same. Such a $P'$ clearly satisfies the conditions in Theorem~\ref{cruchar} and it was not obtained previously by filling in a corner element (as $n+1$ has no element to the left of it in $P'$).
\end{proof}

\begin{rem} \normalfont There is plenty of room to strengthen (\ref{monot-cru}), for example, by using the proof of Theorem~\ref{crucial-growth-thm}. In particular, (\ref{casei-impr}) states that $s^{(c)}_n(k,\ell)+s^{(c)}_n(k,\ell-1)\leq s^{(c)}_{n+1}(k,\ell)$ for $k+\ell-3\leq n\leq (k-1)(\ell-2)$. As for the case of $(k-1)(\ell-2)<n<(k-1)(\ell-1)$, we can prove that 
\begin{equation}\label{non-accurate-estimate}s^{(c)}_n(k,\ell)+I(k,\ell-1)\sum_{d=1}^{k-2}I(d+1,\ell)I(k-d,\ell-1)\leq s^{(c)}_{n+1}(k,\ell)\end{equation}
where $I(n,k)$ is given by the square root of the formula (\ref{Stanley-formula}), that is, $I(n,k)$ is the number of SYT of the $(\ell-1)\times (k-1)$ rectangular shape. Indeed, $I(d+1,\ell)$ (resp., $I(k-d,\ell-1)$) is the number of possibilities to choose a SYT in case (ii)(a) (resp., (ii)(b)) in the proof of Theorem~\ref{crucial-growth-thm} (independently from each other), which gives an estimation for the number of possible $P$ after summing over all $d$. Moreover, for each such a choice of $P$ we can make several choices for $Q$ depending on the shape of $P$. All these choices can be estimated from below by dropping the last row and considering $I(k,\ell-1)$ $(\ell-2)\times k-1$ rectangular SYT\footnote{We can produce an explicit formula for the number of $Q$ for any fixed $d$ hence introducing an extra factor under the sum in (\ref{non-accurate-estimate}) and improving the inequality. However, this is not done for the sake of keeping simplicity.}, which gives (\ref{non-accurate-estimate}). Finally, we note that while our estimates are generally much better than $s^{(c)}_n(k,\ell)+1\leq s^{(c)}_{n+1}(k,\ell)$, they are still far from being accurate. For example, our inequalities imply $15=s^{(c)}_5(4,3)\geq 6$, $64=s^{(c)}_6(4,4)\geq 21$ and $378=s^{(c)}_7(4,4)\geq 89$, and they become much worse with growing $k$ and $\ell$. \end{rem}

In this paper, we use the following properties of the RSK correspondence, where $A^T$ is the transpose of a SYT $A$, and recall that $\epsilon$ is the evacuation map.
\begin{lemma}[\cite{Lee, Sagan}]\label{RSK-properties}
Let $\sigma\in S_n$ and $(P,Q)$ be the pair of SYT corresponding to $\sigma$.  Then,
\begin{itemize}
\item $(P^T, \epsilon (Q)^T)$ is the pair of SYT corresponding to the reverse $r(\sigma)$.
\item $(Q,P)$ is the pair of SYT corresponding to the inverse $\sigma^{-1}$.
\item $(\epsilon(P)^T, Q^T)$ is the pair of SYT corresponding to the complement $c(\sigma)$.
\end{itemize}
\end{lemma}

Next, we characterize $(k,\ell)$-left-, top- and bottom-crucial permutations.

\begin{thm} \label{left-top-bottom}
Let $\sigma\in S_n$ and $(P,Q)$ be the pair of SYT corresponding to $\sigma$. Then,  $\sigma$ is $(k,\ell)$-left-, top- or bottom-crucial if and only if the following holds:
\begin{itemize}

\item the number of columns in $P$ (and $Q$) is $k-1$;
\item the number of rows in $P$ (and $Q$) is $\ell-1$;
\item {\bf for $(k,\ell)$-left-crucial:} $P$  contains an increasing sequence of elements $p_{\ell-1, 1}$,  $p_{i_2,2},$ \ldots, $p_{i_{k-1}, k-1}$ for some $i_2, \ldots, i_{k-1}$;

{\bf for $(k,\ell)$-top-crucial:}  $Q$  contains an increasing sequence of elements $q_{1, k-1}$, $q_{2,j_2}$, \ldots, $q_{\ell-1, j_{\ell-1}}$ for some $j_2, \ldots, j_{\ell-1}$;

{\bf for $(k,\ell)$-bottom-crucial:} $Q$  contains an increasing sequence of elements $q_{\ell-1, 1}$,  $q_{i_2,2},$ \ldots, $q_{i_{k-1}, k-1}$ for some $i_2, \ldots, i_{k-1}$.
\end{itemize}
\end{thm}

\begin{proof}
Remark~\ref{longest-remark} gives conditions on the number of rows and columns in $P$ and $Q$.

By Lemma~\ref{crucial-types}, $\sigma$ is $(k,\ell)$-left-crucial if and only if $r(\sigma)$ is  $(\ell, k)$-right-crucial. By Lemma~\ref{RSK-properties}, $(P^T, \epsilon (Q)^T)$ is the pair of SYT corresponding to  $r(\sigma)$. Hence, Theorem~\ref{cruchar} gives the required characterization of $(k,\ell)$-left-crucial permutations. 

By Lemma~\ref{crucial-types}, $\sigma$ is $(k,\ell)$-top-crucial if and only if $\sigma^{-1}$ is  $(k, \ell)$-right-crucial. By Lemma~\ref{RSK-properties}, $(Q,T)$ is the pair of SYT corresponding to  $\sigma^{-1}$. Hence, Theorem~\ref{cruchar} gives the required characterization of $(k,\ell)$-top-crucial permutations. 

By Lemma~\ref{crucial-types}, $\sigma$ is $(k,\ell)$-bottom-crucial if and only if $(c(\sigma))^{-1}$ is  $(\ell, k)$-right-crucial. By Lemma~\ref{RSK-properties}, $(Q^T,\epsilon(P)^T)$  is the pair of SYT corresponding to  $(c(\sigma))^{-1}$. Hence, Theorem~\ref{cruchar} gives the required characterization of $(k,\ell)$-bottom-crucial permutations. 
\end{proof}

\section{Characterization of $(k,\ell)$-bicrucial permutations via RSK correspondence}\label{RSK-characterization-bicrucial}

The following theorem is an immediate corollary of Theorems~\ref{cruchar} and \ref{left-top-bottom} and the fact that the set of $(k,\ell)$-bicrucial permutations is the intersection of the sets of $(k,\ell)$-right-crucial and $(k,\ell)$-left-crucial permutations. 

\begin{thm}  \label{bicruchar}
Let $\sigma\in S_n$ and $(P,Q)$ be the pair of SYT corresponding to $\sigma$. Then,  $\sigma$ is $(k,\ell)$-bicrucial if and only if the following holds:
\begin{itemize}
\item the number of columns in $P$ (and $Q$) is $k-1$;
\item the number of rows in $P$ (and $Q$) is $\ell-1$;
\item $P$  contains an increasing sequence of elements $p_{1, k-1}, p_{2,j_2}, \ldots, p_{\ell-1, j_{\ell-1}}$ for some $j_2, \ldots, j_{\ell-1}$;
\item $P$  contains an increasing sequence of elements $p_{\ell-1, 1}, p_{i_2,2}, \ldots, p_{i_{k-1}, k-1}$ for some $i_2, \ldots, i_{k-1}$.
\end{itemize}
\end{thm}

Our next theorem characterizes minimal $(k,\ell)$-bicrucial permutations, and it is the main tool for enumeration of these permutations in Corollary~\ref{min-bicrucial-count-cor}.

\begin{thm}\label{min-bicrucial-char-thm}
For $k > \ell\geq 3$,  any minimal $(k,\ell)$-bicrucial permutation $\sigma$ is of length $n = k + 2\ell -5$. Moreover, if $(P,Q)$ is a pair of SYT corresponding to $\sigma$, then $P$ is given by 
%$$P = \begin{array}{llllll}
%1 & \ell & \ell+1 & \cdots & k+\ell -3 \\
%2 & k+\ell -2 & ~ & ~ & ~ \\
%\vdots & \vdots & ~ & ~ & ~ \\
%\ell-1 &  2\ell+k-5 & ~ & ~ & ~ \\
%\end{array}$$

\begin{center}
\begin{tikzpicture}[scale=0.4]

\draw [line width=0.5](0,5)--(7.3,5);
\draw [line width=0.5](8.7,5)--(12,5);
\draw [line width=0.5](0,4)--(7.3,4);
\draw [line width=0.5](8.7,4)--(12,4);
\draw [line width=0.5](0,3)--(5,3);
\draw [line width=0.5](0,1)--(5,1);
\draw [line width=0.5](0,0)--(5,0);

\draw [line width=0.5](0,2.7)--(0,5);
\draw [line width=0.5](0,0)--(0,1.3);
\draw [line width=0.5](1.6,2.7)--(1.6,5);
\draw [line width=0.5](1.6,0)--(1.6,1.3);
\draw [line width=0.5](5,0)--(5,1.3);
\draw [line width=0.5](5,2.7)--(5,5);
\draw [line width=0.5](7,4)--(7,5);
\draw [line width=0.5](9,4)--(9,5);
\draw [line width=0.5](12,4)--(12,5);

\node [below] at (0.8,5.1){{\tiny $1$}};
\node [below] at (3,5.1){{\tiny $\ell$}};
\node [below] at (6,5.1){{\tiny $\ell+1$}};
\node [below] at (8,5){{\tiny $\cdots$}};
\node [below] at (10.5,5.1){{\tiny $k+\ell-3$}};
\node [below] at (3.1,4.1){{\tiny $k+\ell-2$}};
\node [below] at (0.8,4.1){{\tiny $2$}};
\node [below] at (0.8,3.3){{\tiny $\vdots$}};
\node [below] at (3.2,3.3){{\tiny $\vdots$}};
\node [below] at (0.8,1.1){{\tiny $\ell-1$}};
\node [below] at (3.3,1.1){{\tiny $k+2\ell-5$}};

\end{tikzpicture}
\end{center}

\noindent
For $k=\ell$,  any minimal $(k,k)$-bicrucial permutation $\sigma$ is of length $n = 3k -5$. Moreover, if $(P,Q)$ is a pair of SYT corresponding to $\sigma$, then $P$, or $P^T$, is given by 
%$$P = \begin{array}{llllll}
%1 & \ell & \ell+1 & \cdots & k+\ell -3 \\
%2 & k+\ell -2 & ~ & ~ & ~ \\
%\vdots & \vdots & ~ & ~ & ~ \\
%\ell-1 &  2\ell+k-5 & ~ & ~ & ~ \\
%\end{array}$$

\begin{center}
\begin{tikzpicture}[scale=0.4]

\draw [line width=0.5](0,5)--(7.3,5);
\draw [line width=0.5](8.7,5)--(12,5);
\draw [line width=0.5](0,4)--(7.3,4);
\draw [line width=0.5](8.7,4)--(12,4);
\draw [line width=0.5](0,3)--(5,3);
\draw [line width=0.5](0,1)--(5,1);
\draw [line width=0.5](0,0)--(5,0);

\draw [line width=0.5](0,2.7)--(0,5);
\draw [line width=0.5](0,0)--(0,1.3);
\draw [line width=0.5](1.6,2.7)--(1.6,5);
\draw [line width=0.5](1.6,0)--(1.6,1.3);
\draw [line width=0.5](5,0)--(5,1.3);
\draw [line width=0.5](5,2.7)--(5,5);
\draw [line width=0.5](7,4)--(7,5);
\draw [line width=0.5](9,4)--(9,5);
\draw [line width=0.5](12,4)--(12,5);

\node [below] at (0.8,5.1){{\tiny $1$}};
\node [below] at (3,5.1){{\tiny $k$}};
\node [below] at (6,5.1){{\tiny $k+1$}};
\node [below] at (8,5){{\tiny $\cdots$}};
\node [below] at (10.5,5.1){{\tiny $2k-3$}};
\node [below] at (3.1,4.1){{\tiny $2k-2$}};
\node [below] at (0.8,4.1){{\tiny $2$}};
\node [below] at (0.8,3.3){{\tiny $\vdots$}};
\node [below] at (3.2,3.3){{\tiny $\vdots$}};
\node [below] at (0.8,1.1){{\tiny $k-1$}};
\node [below] at (3.3,1.1){{\tiny $3k-5$}};

\end{tikzpicture}
\end{center}

\end{thm}

\begin{proof}
By Theorem~\ref{bicruchar}, $P$ has $k-1$ columns and $\ell-1$ rows. Moreover, $P$  contains an increasing sequence of elements $p_{1, k-1}, p_{2,j_2}, \ldots, p_{\ell-1, j_{\ell-1}}$ for some $j_2, \ldots, j_{\ell-1}$ and an increasing sequence of elements $p_{\ell-1, 1}, p_{i_2,2}, \ldots, p_{i_{k-1}, k-1}$ for some $i_2, \ldots, i_{k-1}$. 

Note that the minimal possible value of $p_{1,k-1}$ in any SYT $P$ is $k-1$ and the minimal possible value of $p_{\ell-1,1}$ in any SYT $P$ is $\ell-1$.

If $p_{\ell-1,1} = \ell -1$, then the  increasing sequence of elements $p_{\ell-1, 1}$, $p_{i_2,2}$, \ldots, $p_{i_{k-1}, k-1}$ with the smallest possible values has entries $\ell -1, \ell, \ldots, k+ \ell - 3$ and $i_2= i_3=\cdots= i_{k-1}=1$. Note that in this case we have $p_{1,k-1} = k+ \ell -3$. Then the  longest increasing sequence $p_{1, k-1}, p_{2,j_2}, \ldots, p_{\ell-1, j_{\ell-1}}$  with the smallest possible values has entries $k+ \ell -3, k+\ell-2, \ldots, k+ 2\ell - 5$ and $j_2=j_3= \cdots= j_{\ell-1}=2$. This gives us a SYT corresponding to an irreducible (by Theorem~\ref{thm-min-bicrucial-structure}) $(k,\ell)$-bicrucial permutation of length $k+2\ell - 5$.

Similarly, if $p_{1,k-1} = k-1$, then the  increasing sequence of elements $p_{1, k-1}, p_{2,j_2}, \ldots, p_{\ell-1, j_{\ell -1}}$ with the smallest possible values has entries $k -1, k, \ldots, k+ \ell - 3$ and $j_2=j_3= \cdots= j_{\ell-1}=1$. Note that in this case we have $p_{\ell-1,1} = k+ \ell -3$. Then, the  longest increasing sequence $p_{\ell-1, 1}, p_{i_2,2}, \ldots, p_{i_{k-1}, k-1}$  with the smallest possible values has entries $k+ \ell -3, k+\ell-2, \ldots, 2k+ \ell - 5$ and $i_2=i_3= \cdots= i_{k-1}=2$. This gives us a SYT corresponding to an irreducible (by Theorem~\ref{thm-min-bicrucial-structure}) $(k,\ell)$-bicrucial permutation of length $2k+\ell - 5$.

We see that if $k > \ell$, then the case of $p_{\ell-1,1} = \ell -1$ gives a shorter permutation with the respective tableau $P$ shown below, where we underline the first sequence and indicate the second one in bold: 
%$$P = \begin{array}{llllll}
%1 & \textbf{l} & \textbf{l+1} & \cdots & \textbf{\textit{k+l-3}} \\
%2 & \textit{k+l-2} & ~ & ~ & ~ \\
%\vdots & \vdots & ~ & ~ & ~ \\
%\textbf{l-1} &  \textit{2l+k-5} & ~ & ~ & ~ \\
%\end{array}$$

\begin{center}
\begin{tikzpicture}[scale=0.4]

\draw [line width=0.5](0,5)--(7.3,5);
\draw [line width=0.5](8.7,5)--(12.2,5);
\draw [line width=0.5](0,4)--(7.3,4);
\draw [line width=0.5](8.7,4)--(12.2,4);
\draw [line width=0.5](0,3)--(5,3);
\draw [line width=0.5](0,1)--(5,1);
\draw [line width=0.5](0,0)--(5,0);

\draw [line width=0.5](0,2.7)--(0,5);
\draw [line width=0.5](0,0)--(0,1.3);
\draw [line width=0.5](1.7,2.7)--(1.7,5);
\draw [line width=0.5](1.7,0)--(1.7,1.3);
\draw [line width=0.5](5,0)--(5,1.3);
\draw [line width=0.5](5,2.7)--(5,5);
\draw [line width=0.5](7,4)--(7,5);
\draw [line width=0.5](9,4)--(9,5);
\draw [line width=0.5](12.2,4)--(12.2,5);

\node [below] at (0.8,5.1){{\tiny $1$}};
\node [below] at (3,5.1){{\tiny $\boldsymbol{\ell}$}};
\node [below] at (6,5.1){{\tiny $\boldsymbol{\ell+1}$}};
\node [below] at (8,5){{\tiny $\cdots$}};
\node [below] at (10.6,5.1){{\tiny \underline{$\boldsymbol{k+\ell-3}$}}};
\node [below] at (3.2,4.1){{\tiny \underline{$k+\ell-2$}}};
\node [below] at (0.8,4.1){{\tiny $2$}};
\node [below] at (0.8,3.3){{\tiny $\vdots$}};
\node [below] at (3.2,3.3){{\tiny $\vdots$}};
\node [below] at (0.9,1.1){{\tiny $\boldsymbol{\ell-1}$}};
\node [below] at (3.3,1.1){{\tiny \underline{$k+2\ell-5$}}};

\end{tikzpicture}
\end{center}
If $k = \ell$ then both of the cases of $p_{\ell-1,1} = \ell -1$ and $p_{1,k-1} = k-1$ give permutations of the same length and the respective $P$'s in these cases are obtained from each other by transposition that has no fixed points.
\end{proof}

\begin{cor}\label{min-bicrucial-count-cor} Let $k\geq \ell$. The number of minimal $(k,\ell)$-bicrucial permutations is given by 
$$\frac{(1+\delta_{k,\ell})(k+2\ell-5)!}{(k+\ell-3)(k+\ell-4)(\ell-1)!(\ell-2)!(k-3)!}$$
where $\delta_{k,\ell}$ is the Kronecker delta, namely,
$$\delta_{k,\ell}=\left\{\begin{array}{cc}0 & \mbox{ if }k\neq\ell\\ 1 & \mbox{ if }k=\ell\end{array}\right..$$
\end{cor}

\begin{proof} By Theorem~\ref{min-bicrucial-char-thm} we know the shape of the SYT corresponding to a $(k,\ell)$-bicrucial permutation $\sigma$, and we only need to apply the hook length formula to count the number of possible SYT $Q$.  The tableau listing the hook length of each cell in the respective Young diagram is

\begin{center}
\begin{tikzpicture}[scale=0.4]

\draw [line width=0.5](0,0)--(0,1.3);
\draw [line width=0.5](0,2.7)--(0,6);
\draw [line width=0.5](3,0)--(3,1.3);
\draw [line width=0.5](3,2.7)--(3,6);
\draw [line width=0.5](6,0)--(6,1.3);
\draw [line width=0.5](6,2.7)--(6,6);
\draw [line width=0.5](8,5)--(8,6);
\draw [line width=0.5](10,5)--(10,6);
\draw [line width=0.5](12,5)--(12,6);
\draw [line width=0.5](14,5)--(14,6);

\draw [line width=0.5](0,0)--(6,0);
\draw [line width=0.5](0,1)--(6,1);
\draw [line width=0.5](0,3)--(6,3);
\draw [line width=0.5](0,4)--(6,4);
\draw [line width=0.5](11.7,5)--(14,5);
\draw [line width=0.5](0,5)--(10.3,5);
\draw [line width=0.5](11.7,6)--(14,6);
\draw [line width=0.5](0,6)--(10.3,6);

\node [below] at (1.5,6.1){{\tiny $k+\ell-3$}};
\node [below] at (1.3,5.1){{\tiny $\ell-1$}};
\node [below] at (1.3,4.1){{\tiny $\ell-2$}};
\node [below] at (1.4,3.3){{\tiny $\vdots$}};
\node [below] at (1.3,1.1){{\tiny $2$}};

\node [below] at (4.5,6.1){{\tiny $k+\ell-4$}};
\node [below] at (4.3,5.1){{\tiny $\ell-2$}};
\node [below] at (4.3,4.1){{\tiny $\ell-3$}};
\node [below] at (4.4,3.3){{\tiny $\vdots$}};
\node [below] at (4.3,1.1){{\tiny $1$}};

\node [below] at (7,6.1){{\tiny $k-3$}};
\node [below] at (9,6.1){{\tiny $k-4$}};
\node [below] at (11,6){{\tiny $\cdots$}};
\node [below] at (13,6.1){{\tiny $1$}};

\end{tikzpicture}
\end{center}

\noindent
from where the result follows by noticing the factor of 2 in the case of $k=\ell$ (both $P$ and $P^T$ cannot be of the required form). 
\end{proof}

As an immediate consequence of Corollary~\ref{min-bicrucial-count-cor} we have the following result. 

\begin{cor}\label{min-k3-bicrucial-count-cor} 
Let $k\geq 3$. The number of minimal $(k,3)$-bicrucial permutations is given by 
$$\frac{(1+\delta_{k,3})(k+1)(k-2)}{2}$$
where $\delta_{k,\ell}$ is the Kronecker delta. For $k>3$ the sequence begins with $5, 9, 14, 20, 27, 35,\ldots$ and this is the sequence $A000096$ in \cite{oeis} with many combinatorial interpretations. 
\end{cor}

\begin{rem}\label{irr-vs-min} \normalfont It follows from the proof of Theorem~\ref{min-bicrucial-char-thm} that irreducible $(k,\ell)$-bicrucial permutations are not necessarily minimal $(k,\ell)$-bicrucial permutations as for $k>\ell$ irreducible $(k,\ell)$-bicrucial permutations of length $2k+\ell-5$ exist.  \end{rem}

\begin{thm}\label{bicrucial-growth-thm} For $k\geq\ell$ and $k+2\ell-5\leq n < (k-1)(\ell-1)$, we have the following monotone property for the number of bicrucial permutations:
\begin{equation}\label{monot-bicru}s^{(b)}_n(k,\ell)< s^{(b)}_{n+1}(k,\ell).\end{equation} 
\end{thm}

\begin{proof} Our proof is very similar to the alternative proof of Theorem~\ref{crucial-growth-thm}, so we just provide a sketch of it. 

Note that for $k+2\ell-5<n<(k-1)(\ell-1)-2$ we can alway construct a SYT $P$ with at least two corners satisfying the conditions in Theorem~\ref{bicruchar} (e.g. by starting with a $P$ corresponding to a minimal $(k,\ell)$-bicrucial permutation and then introducing new largest elements, one by one, in corners). That means that we can extend $P$ in more than one way by placing $n+1$ in one of the corners, and such an extension is clearly injective and results in a SYT $P'$  satisfying the conditions in Theorem~\ref{bicruchar}. In the remaining cases we have exactly one corner and placing $n+1$ in that corner results in a SYT satisfying the conditions in Theorem~\ref{bicruchar}. If $n\in\{(k-1)(\ell-1)-2,(k-1)(\ell-1)-1\}$ then we can apply the construction in case (ii) in the proof of Theorem~\ref{crucial-growth-thm} to get an extra $(k,\ell)$-crucial permutation of length $n+1$. 

Finally, if $n=k+2\ell-5$ then note that $p_{\ell-1,2}=n$. Suppose that $\ell>3$. We can extend $P$ to $P'$ by letting the corner $p'_{2,3}=n$, the element $p'_{\ell-1,2}=n+1$, and keeping all other elements of $P$ the same. Such a $P'$ clearly satisfies the conditions in Theorem~\ref{bicruchar} and it was not obtained previously by filling in the corner element (as $n+1$ is in the bottom row now unlike the other case). If $\ell=3$ then we get an extra extension by letting the corner element $p'_{2,3}=n+1$ and swapping the elements $p_{2,2}=n$ and $p_{1,k-1}$ while keeping all other elements of $P$ the same, which is clearly an injective operation. 
\end{proof}

\begin{rem} \normalfont If follows directly from the proof of Theorem~\ref{bicrucial-growth-thm} that $$2s^{(b)}_n(k,\ell)\leq s^{(b)}_{n+1}(k,\ell).$$ \end{rem}

\section{Characterization of $(k,\ell)$-top-right-crucial permutations via RSK correspondence}\label{RSK-characterization-top-right-crucial}

The following theorem is an immediate corollary of Theorems~\ref{cruchar} and \ref{left-top-bottom} and the fact that the set of $(k,\ell)$-top-right-crucial permutations is the intersection of the sets of $(k,\ell)$-right-crucial and $(k,\ell)$-top-crucial permutations. 

\begin{thm} \label{top-right-cruchar}
Let $\sigma\in S_n$ and $(P,Q)$ be the pair of SYT corresponding to $\sigma$. Then,  $\sigma$ is $(k,\ell)$-top-right-crucial if and only if the following holds:
\begin{itemize}
\item the number of columns in $P$ (and $Q$) is $k-1$;
\item the number of rows in $P$ (and $Q$) is $\ell-1$;
\item $P$  contains an increasing sequence of elements $p_{1, k-1}, p_{2,j_2}, \ldots, p_{\ell-1, j_{\ell-1}}$ for some $j_2, \ldots, j_{\ell-1}$;
\item $Q$  contains an increasing sequence of elements $q_{1, k-1}$, $q_{2,j_2}$, \ldots, $q_{\ell-1, j_{\ell-1}}$ for some $j_2, \ldots, j_{\ell-1}$.
\end{itemize}
\end{thm}

As a corollary to Theorem~\ref{top-right-cruchar}, we have the following result.

\begin{cor}\label{k3-top-right-crucial-count} The number of $(k,3)$-top-right-crucial permutations of length $n\geq k+1$, for $k\geq 3$, is given by 
$$
s^{(tr)}_n(k,3)=\left(\frac{2k-n}{n-k}{n-1\choose k}\right)^2.
$$
Moreover, for $2\leq i\leq k-1$, the number of $(k,3)$-right-crucial permutations of $i$-th smallest length $k+i-1$ is given by
$$
\left(\frac{k-i+1}{i-1}{k+i-2\choose k}\right)^2
$$
and this number is $0$ for $i\geq k$. 
%By symmetry, replacing $k$ by $\ell$ in the formulas, we obtain the results for $(3,\ell)$-crucial permutations for $\ell\geq 3$.  
\end{cor}

\begin{proof} Our proof follows exactly the same steps as those in the proof of Corollary~\ref{k3-crucial-count} except we note that the number of choices for $Q$ is now the same as the number of choices for $P$ and hence is given by (\ref{SYT-P-formula-k3}).\end{proof}

\begin{rem}\label{rem-square-numbers}\normalfont  The fact that $P$ and $Q$ have the same restrictions for $(k,\ell)$-top-right-crucial permutations implies that the number $s^{(tr)}_n(k,\ell)$ of such permutations of length $n$ is the sum of $t$ square numbers, where $t$ is the number of Young diagrams with $n$ squares, $k-1$ columns and $\ell-1$ rows (equivalently, $t$ is the number of partitions of $n$ into $\ell-1$ parts with the largest part of size $k-1$). In particular, the number of minimal and next minimal $(k,\ell)$-top-right-crucial permutations, respectively, $s^{(tr)}_{k+\ell-3}(k,\ell)$ and  $s^{(tr)}_{k+\ell-2}(k,\ell)$, are square numbers for any $k$ and $\ell$ as there are unique Young diagrams in question. By Theorem~\ref{thm-min-top-right-crucial},  $s^{(tr)}_{k+\ell-3}(k,\ell)=1$ and in Theorem~\ref{next-smallest-top-right-enum} we will show that   $s^{(tr)}_{k+\ell-2}(k,\ell)=(k+\ell-4)^2$.  See Theorem~\ref{quadro-square-sum} for a relevant result on quadrocrucial permutations. \end{rem}

\begin{thm}\label{next-smallest-top-right-enum} For $k,\ell\geq 3$, the number of next minimal $(k,\ell)$-top-right-crucial permutations is $s^{(tr)}_{k+\ell-2}(k,\ell)=(k+\ell-4)^2$.  \end{thm}
 
 \begin{proof} First note that by Theorem~\ref{top-right-cruchar}, the Young diagram of any $(k,\ell)$-top-right-crucial permutation of length $k+\ell-2$ has $k-1$ squares in the first row, two squares in the second row and one square in rows $3, 4,\ldots,\ell-1$. We wish to count the number of ways to build $P$, which is the number of possible SYT of the shape satisfying the ``increasing sequence of elements'' condition in Theorem~\ref{top-right-cruchar}.
 
 Any SYT having the shape in question has an $i$, $1\leq i\leq k-1$, such that the first row in $P$ begins with $1,2,\ldots,i$ and $p_{2,1}=i+1$. If $i<k-1$ then, by Theorem~\ref{top-right-cruchar},  we must have $p_{1,i+1}<p_{1,i+2}<\cdots <p_{1,k-1}<p_{2,2}<p_{3,1}<p_{4,1}<\cdots <p_{\ell-1,1}$
 and hence there is a unique way to build such a $P$ (where, in particular, $p_{1,i+1}=i+2$). That far, we found $k-2$ ways to construct $P$. However, if $i=k-1$ (and $p_{2,1}=k$) then $p_{2,2}$ can be any element in $\{k+1,k+2,\ldots,k+\ell-2\}$ giving us additional $\ell-2$ possibilities to build $P$. The claimed result now follows from $k-2+\ell-2=k+\ell-4$ and the observation that by Theorem~\ref{top-right-cruchar}, the number of ways to build $P$ is the same as that to build~$Q$.  \end{proof}
 
Following the same arguments as in the proof of Theorem~\ref{crucial-growth-thm} related to $(k,\ell)$-crucial permutations, except always choosing for $Q$  the SYT of the form in question obtained by filling in squares with consecutively increasing numbers from left to right and from top to bottom (that would satisfy the condition on $Q$ in Theorem~\ref{top-right-cruchar}), we obtain the following result.

\begin{thm}\label{top-right-crucial-growth-thm} For $k+\ell-3\leq n < (k-1)(\ell-1)$, we have the following monotone property for the number of top-right-crucial permutations:
\begin{equation}\label{monot-trcru}s^{(tr)}_n(k,\ell)< s^{(tr)}_{n+1}(k,\ell).\end{equation} 
\end{thm}

\section{Characterization of $(k,\ell)$-tricrucial permutations via RSK correspondence}\label{RSK-characterization-tricrucial}

Recall that a permutation is $(k,\ell)$-tricrucial if it is $(k,\ell)$-bicrucial and any of its extensions from above results in a permutation containing an occurrence of $i_k$ or $d_{\ell}$. The following theorem is an immediate corollary of Theorems~\ref{left-top-bottom} and \ref{bicruchar} and the fact that the set of $(k,\ell)$-tricrucial permutations is the intersection of the sets of $(k,\ell)$-bicrucial and $(k,\ell)$-top-crucial permutations. 

\begin{thm}  \label{tricruchar}
Let $\sigma\in S_n$ and $(P,Q)$ be the pair of SYT corresponding to $\sigma$. Then,  $\sigma$ is $(k,\ell)$-tricrucial if and only if the following holds:
\begin{itemize}
\item the number of columns in $P$ (and $Q$) is $k-1$;
\item the number of rows in $P$ (and $Q$) is $\ell-1$;
\item $P$  contains an increasing sequence of elements $p_{1, k-1}, p_{2,j_2}, \ldots, p_{\ell-1, j_{\ell-1}}$ for some $j_2, \ldots, j_{\ell-1}$;
\item $P$  contains an increasing sequence of elements $p_{\ell-1, 1}, p_{i_2,2}, \ldots, p_{i_{k-1}, k-1}$ for some $i_2, \ldots, i_{k-1}$;
\item $Q$  contains an increasing sequence of elements $q_{1, k-1}$, $q_{2,j_2}$, \ldots, $q_{\ell-1, j_{\ell-1}}$ for some $j_2, \ldots, j_{\ell-1}$.
\end{itemize}
\end{thm}

\begin{thm}\label{min-tricrucial-char-thm}
For $k > \ell$,  any minimal $(k,\ell)$-tricrucial permutation $\sigma$ is of length $n = k + 2\ell -5$. Moreover, if $(P,Q)$ is a pair of SYT corresponding to $\sigma$, then $P$ is given by 

\begin{center}
\begin{tikzpicture}[scale=0.4]

\draw [line width=0.5](0,5)--(7.3,5);
\draw [line width=0.5](8.7,5)--(12,5);
\draw [line width=0.5](0,4)--(7.3,4);
\draw [line width=0.5](8.7,4)--(12,4);
\draw [line width=0.5](0,3)--(5,3);
\draw [line width=0.5](0,1)--(5,1);
\draw [line width=0.5](0,0)--(5,0);

\draw [line width=0.5](0,2.7)--(0,5);
\draw [line width=0.5](0,0)--(0,1.3);
\draw [line width=0.5](1.6,2.7)--(1.6,5);
\draw [line width=0.5](1.6,0)--(1.6,1.3);
\draw [line width=0.5](5,0)--(5,1.3);
\draw [line width=0.5](5,2.7)--(5,5);
\draw [line width=0.5](7,4)--(7,5);
\draw [line width=0.5](9,4)--(9,5);
\draw [line width=0.5](12,4)--(12,5);

\node [below] at (0.8,5.1){{\tiny $1$}};
\node [below] at (3,5.1){{\tiny $\ell$}};
\node [below] at (6,5.1){{\tiny $\ell+1$}};
\node [below] at (8,5){{\tiny $\cdots$}};
\node [below] at (10.5,5.1){{\tiny $k+\ell-3$}};
\node [below] at (3.1,4.1){{\tiny $k+\ell-2$}};
\node [below] at (0.8,4.1){{\tiny $2$}};
\node [below] at (0.8,3.3){{\tiny $\vdots$}};
\node [below] at (3.2,3.3){{\tiny $\vdots$}};
\node [below] at (0.8,1.1){{\tiny $\ell-1$}};
\node [below] at (3.3,1.1){{\tiny $k+2\ell-5$}};

\end{tikzpicture}
\end{center}
and  $q_{1, k-1} < q_{2,2}$ in  $Q$.

\noindent
For $k=\ell$,  any minimal $(k,k)$-tricrucial permutation $\sigma$ is of length $n = 3k -5$. Moreover, if $(P,Q)$ is a pair of SYT corresponding to $\sigma$, then $P$, or $P^T$, is given by 

\begin{center}
\begin{tikzpicture}[scale=0.4]

\draw [line width=0.5](0,5)--(7.3,5);
\draw [line width=0.5](8.7,5)--(12,5);
\draw [line width=0.5](0,4)--(7.3,4);
\draw [line width=0.5](8.7,4)--(12,4);
\draw [line width=0.5](0,3)--(5,3);
\draw [line width=0.5](0,1)--(5,1);
\draw [line width=0.5](0,0)--(5,0);

\draw [line width=0.5](0,2.7)--(0,5);
\draw [line width=0.5](0,0)--(0,1.3);
\draw [line width=0.5](1.6,2.7)--(1.6,5);
\draw [line width=0.5](1.6,0)--(1.6,1.3);
\draw [line width=0.5](5,0)--(5,1.3);
\draw [line width=0.5](5,2.7)--(5,5);
\draw [line width=0.5](7,4)--(7,5);
\draw [line width=0.5](9,4)--(9,5);
\draw [line width=0.5](12,4)--(12,5);

\node [below] at (0.8,5.1){{\tiny $1$}};
\node [below] at (3,5.1){{\tiny $k$}};
\node [below] at (6,5.1){{\tiny $k+1$}};
\node [below] at (8,5){{\tiny $\cdots$}};
\node [below] at (10.5,5.1){{\tiny $2k-3$}};
\node [below] at (3.1,4.1){{\tiny $2k-2$}};
\node [below] at (0.8,4.1){{\tiny $2$}};
\node [below] at (0.8,3.3){{\tiny $\vdots$}};
\node [below] at (3.2,3.3){{\tiny $\vdots$}};
\node [below] at (0.8,1.1){{\tiny $k-1$}};
\node [below] at (3.3,1.1){{\tiny $3k-5$}};

\end{tikzpicture}
\end{center}
and in $Q$, $q_{1, k-1} < q_{2,2}$ in the former case and $q_{1,k-1} < q_{2,j} < q_{3,1}$, for some $j \in \{1, \ldots k-1 \}$, in the latter case. 

\end{thm}

\begin{proof}
Any tricrucial permutation is a bicrucial permutation, and in Theorem~\ref{min-bicrucial-char-thm} we have a characterization for $(k,\ell)$-bicrucial permutations, which gives, in particular, the claimed shape and filling of $P$ and the shape of $Q$. 

By Theorem~\ref{left-top-bottom}, a $(k,\ell)$-bicrucial permutation is $(k,\ell)$-top-crucial (and, consequently, $(k,\ell)$-tricrucial) if the corresponding SYT $Q$ contains an increasing sequence of elements $q_{1, k-1}$, $q_{2,j_2}$, \ldots, $q_{\ell-1, j_{\ell-1}}$ for some $j_2, \ldots, j_{\ell-1}$.

If $k>\ell$, then SYT $Q$ has the unique shape (presented in the statement of the theorem). Since in any SYT every column is increasing, the presence of an increasing sequence of elements $q_{1, k-1}$, $q_{2,j_2}$, \ldots, $q_{\ell-1, j_{\ell-1}}$ is equivalent to the requirement  $q_{1, k-1} < q_{2,2}$.

If $k = \ell$ then there is another possible shape of $Q$ (the transposition of the shape in the statement of the theorem). In this case, to guarantee the presence of an increasing sequence of elements $q_{1, k-1}$, $q_{2,j_2}$, \ldots, $q_{\ell-1, j_{\ell-1}}$ it is sufficient to check that $q_{1,k-1} < q_{2,j} < q_{3,1}$  for some $j \in \{1, \ldots k-1 \}$.
\end{proof}

\begin{cor}\label{min-k3-tricrucial-count-cor} 
For $k> 3$ (resp., $k=3$), the number of minimal $(k,3)$-tricrucial permutations is $k-1$ (resp., $4$). 
\end{cor}

\begin{proof} Referring to Theorem~\ref{min-tricrucial-char-thm}, we derive the following facts. For $k\geq 3$, the length of permutations in question is $k+1$. For $k>3$, we have the unique choice for $P$, and in $Q$ (having two rows) we must have $q_{1,1}=1$ and $q_{2,2}=k+1$ (the maximum element), and any of the remaining $k-1$ elements can be in square $q_{2,1}$ hence giving the result. If $k=3$, then both SYTs of the $2\times 2$ shape are valid choices for $P$ and $Q$, hence giving in total 4 choices, as required.   \end{proof}

Following the same arguments as in the proof of Theorem~\ref{bicrucial-growth-thm} related to $(k,\ell)$-bicrucial permutations, except always choosing for $Q$  the SYT of the form in question obtained by filling in squares with consecutively increasing numbers from left to right and from top to bottom (that would satisfy the condition on $Q$ in Theorem~\ref{tricruchar}), we obtain the following result.

\begin{thm}\label{tricrucial-growth-thm} For $k+\ell-3\leq n < (k-1)(\ell-1)$, we have the following monotone property for the number of tricrucial permutations:
\begin{equation}\label{monot-tricru}s^{(tri)}_n(k,\ell)< s^{(tri)}_{n+1}(k,\ell).\end{equation} 
\end{thm}

\section{Characterization of $(k,\ell)$-quadrocrucial permutations via the RSK correspondence}\label{RSK-characterization-quadrocrucial}

Recall that a permutation is $(k,\ell)$-quadrocrucial if it is $(k,\ell)$-tricrucial and any extension from below results in a permutation containing an occurrence of $i_k$ or $d_{\ell}$. The following theorem is an immediate corollary of Theorems~\ref{left-top-bottom} and \ref{tricruchar} and the fact that the set of $(k,\ell)$-bicrucial permutations is the intersection of the sets of $(k,\ell)$-tricrucial and $(k,\ell)$-bottom-crucial permutations. 

\begin{thm} \label{quadrocruchar}
Let $\sigma\in S_n$ and $(P,Q)$ be the pair of SYT corresponding to $\sigma$. Then,  $\sigma$ is $(k,\ell)$-quadrocrucial if and only if the following holds:
\begin{itemize}
\item the number of columns in $P$ (and $Q$) is $k-1$;
\item the number of rows in $P$ (and $Q$) is $\ell-1$;
\item $P$  contains an increasing sequence of elements $p_{1, k-1}, p_{2,j_2}, \ldots, p_{\ell-1, j_{\ell-1}}$ for some $j_2, \ldots, j_{\ell-1}$;
\item $P$  contains an increasing sequence of elements $p_{\ell-1, 1}, p_{i_2,2}, \ldots, p_{i_{k-1}, k-1}$ for some $i_2, \ldots, i_{k-1}$;
\item $Q$  contains an increasing sequence of elements $q_{1, k-1}$, $q_{2,j_2}$, \ldots, $q_{\ell-1, j_{\ell-1}}$ for some $j_2, \ldots, j_{\ell-1}$;
\item $Q$  contains an increasing sequence of elements $q_{\ell-1, 1}$,  $q_{i_2,2},$ \ldots, $q_{i_{k-1}, k-1}$ for some $i_2, \ldots, i_{k-1}$.
\end{itemize}
\end{thm}

\begin{thm}\label{minimal-length-quadro} For $k\geq 3$ and $\ell\geq 2$, minimal $(k,\ell)$-quadrocrucial permutations are of length $k+2\ell-5$. \end{thm}

\begin{proof} Any $(k,\ell)$-quadrocrucial permutation is $(k,\ell)$-tricrucial, and hence, by Theorem~\ref{min-tricrucial-char-thm}, the length of a minimal $(k,\ell)$-quadrocrucial permutation is $\geq  k+2\ell-5$. To complete the proof, we construct a  $(k,\ell)$-quadrocrucial permutation of length $k+2\ell-5$. Consider the permutation $\sigma=\sigma_1\cdots\sigma_{k+2\ell-5}=$
\begin{equation}\label{min-quadrocrucial-perm}
(\ell-1)(\ell-2)\cdots 1\ell(\ell+1)\cdots (\ell+k-4)(k+2\ell-5)(k+2\ell-6)\cdots(k+\ell-3)
\end{equation}
having the $\ell-1$ smallest and $\ell-1$ largest elements consecutively in decreasing order, and the remaining $k-3$ elements in the middle in increasing order. 

Clearly, $\sigma$ avoids $i_k$ and $d_{\ell}$. Moreover, extending $\sigma$ 
\begin{itemize} 
\item to the right by an element $\leq \ell+k-3$  (resp., $>\ell+k-3$) introduces $d_{\ell}$ (resp., $i_k$) involving the elements $k+2\ell-5$, $k+2\ell-6,\ldots,\ell+k-3$ (resp., $\ell-1$, $\ell,\ldots,\ell+k-3$) in $\sigma$;
\item to the left by an element $> \ell-1$  (resp., $\leq \ell-1$) introduces $d_{\ell}$ (resp., $i_k$) involving the elements $1$, $2,\ldots,\ell-1$ (resp., $\ell-1$, $\ell,\ldots,k+\ell-3$) in $\sigma$;
\item from above to the left (resp., right) of the element $k+2\ell-5$  introduces $d_{\ell}$ (resp., $i_k$) involving the elements $k+\ell-3$, $k+\ell-2,\ldots,k+2\ell-5$ (resp., $\ell-1$, $\ell,\ldots,\ell+k-4$, $k+2\ell-5$) in $\sigma$;
\item from below to the right (resp., left) of the element $1$  introduces $d_{\ell}$ (resp., $i_k$) involving the elements $1$, $2,\ldots,\ell-1$ (resp., $1$, $\ell,\ell+1,\ldots,\ell+k-4$, $k+2\ell-5$) in $\sigma$.
\end{itemize}
Therefore, $\sigma$ is (minimal) $(k,\ell)$-quadrocrucial.
\end{proof}

In fact, Theorem~\ref{minimal-length-quadro} is a direct corollary of the following theorem. 

\begin{thm}\label{min-kl-quadro}
There is a unique minimal $(k,3)$-quadrocrucial permutation for any $k> 3$, and there are four minimal $(3,3)$-quadrocrucial permutations. Moreover, there are $1+\delta_{k,\ell}$ minimal $(k,\ell)$-quadrocrucial permutations for $k,\ell>3$, where $\delta_{k,\ell}$ is the Kronecker delta.
\end{thm}

\begin{proof} Since for $k=\ell=3$, $k+2\ell-5=(k-1)(\ell-1)$, by Theorem~\ref{minimal-length-quadro} minimal $(3,3)$-quadrocrucial permutations are precisely maximal $(3,3)$-quadrocrucial permutations, and by (\ref{Stanley-formula}), there are $\left(\frac{4!}{1^1\cdot 2^2 \cdot 3^1}\right)^2=4$ such permutations.

Now, assume that $k>\ell\geq 3$. Any quadrocrucial permutation is necessarily bicrucial, and hence Theorem~\ref{min-bicrucial-char-thm} can be applied to see that there is a unique choice for $P$. Moreover, the conditions on $Q$ for the shape of $P$ in Theorem~\ref{quadrocruchar} imply the unique choice of $Q$ (in fact, $P=Q$), which gives the desired result. Finally, if $k=\ell> 3$ then again Theorem~\ref{min-bicrucial-char-thm} can be applied to see that there are two choices for $P$, for each of which there is a unique choice of $Q$ by Theorem~\ref{quadrocruchar} (again, $P=Q$), that completes our proof.
\end{proof}

Following the same arguments as in the proof of Theorem~\ref{bicrucial-growth-thm} related to $(k,\ell)$-bicrucial permutations, but instead of choosing any $Q$ always mimicking on $Q$ the ways we extend $P$ (to make sure that the extended $Q$ satisfies the conditions in Theorem~\ref{quadrocruchar}), we obtain the following result.

\begin{thm}\label{quadrocrucial-growth-thm} For $k+\ell-3\leq n < (k-1)(\ell-1)$, we have the following monotone property for the number of quadrocrucial permutations:
\begin{equation}\label{monot-qcru}s^{(q)}_n(k,\ell)< s^{(q)}_{n+1}(k,\ell).\end{equation} 
\end{thm}

The following result is similar to the fact discussed in Remark~\ref{rem-square-numbers} for top-right-crucial permutations, and it is true since $P$ and $Q$ have the same restrictions and the number of choices for them for a fixed shape is a square number. 

\begin{thm}\label{quadro-square-sum}
 There exist $m_1, \ldots, m_t \in \mathbb{N}$ such that the number of $(k,\ell)$-quadrocrucial permutations of length $n$
$$s^{(q)}_n(k,\ell) = m_1^2 + \cdots + m_t^2,$$
where $t$ is the number of Young diagrams with $n$ squares, $k-1$ columns and $\ell-1$ rows (equivalently, $t$ is the number of partitions of $n$ into $\ell-1$ parts with the largest part of size $k-1$).
\end{thm}

\begin{figure}
\begin{center}
\begin{tikzpicture}[scale=0.4]

\draw [line width=0.5](0,0)--(8,0);
\draw [line width=0.5](0,2)--(8,2);
\draw [line width=0.5](0,4)--(8,4);
\draw [line width=0.5](0,6)--(8,6);
\draw [line width=0.5](0,8)--(8,8);

\draw [line width=0.5](0,0)--(0,8);
\draw [line width=0.5](2,0)--(2,8);
\draw [line width=0.5](4,0)--(4,8);
\draw [line width=0.5](6,0)--(6,8);
\draw [line width=0.5](8,0)--(8,8);

%decreasing
\draw[fill=black] (0.5,3.5) circle (4pt);
\draw[fill=black] (1,3) circle (4pt);
\draw[fill=black] (1.5,2.5) circle (4pt);

%decreasing
\draw[fill=black] (6.5,5.5) circle (4pt);
\draw[fill=black] (7,5) circle (4pt);
\draw[fill=black] (7.5,4.5) circle (4pt);

%increasing
\draw[fill=black] (2.5,6.5) circle (4pt);
\draw[fill=black] (3,7) circle (4pt);
\draw[fill=black] (3.5,7.5) circle (4pt);

%increasing
\draw[fill=black] (4.5,0.5) circle (4pt);
\draw[fill=black] (5,1) circle (4pt);
\draw[fill=black] (5.5,1.5) circle (4pt);

\node [below] at (5,-0.1){{\tiny $k-2$}};
\node [below] at (3,9.1){{\tiny $k-2$}};
\node [below] at (9.1,5.5){{\tiny $\ell-2$}};
\node [below] at (-1,3.5){{\tiny $\ell-2$}};

\end{tikzpicture}
\end{center}
\vspace{-10mm}
\caption{The structure of the quadrocrucial permutation given by (\ref{counter-examp-quadrocrucial})}\label{structure-of-a-quadrocrucial-perm}
\end{figure}
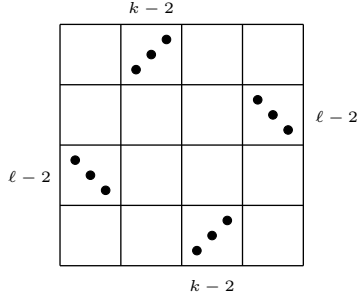

Experimenting with quadrocrucial permutations of smaller lengths, one can notice that any such permutation $\sigma$ has an element $x$ such that replacing this element with either $x(x+1)$ or $(x+1)x$ , and increasing all elements in $\sigma$ larger than $x$ by 1, results in a quadrocrucial permutation of larger length. For example, in the quadrocrucial permutation given by (\ref{min-quadrocrucial-perm}) we can replace the element $\ell$ by $(\ell+1)\ell$ and increase the elements larger than $\ell$ in the original permutation by 1 in order to get a quadrocrucial permutation of length $k+2\ell-4$.  It was tempting to conjecture and prove that this is the case for all quadrocrucial permutations of length less than the maximal length $(k-1)(\ell-1)$. However, this is not the case, as a counterexample is found by considering the quadrocrucial permutation (placed on two lines)
\begin{eqnarray}\label{counter-examp-quadrocrucial}
&& (k+\ell-4)(k+\ell-5)\cdots (k-1)  (k+2\ell-5)(k+2\ell-4)\cdots \\ \nonumber
&& (2k+2\ell-8)12\cdots (k-2) (k+2\ell-6)(k+2\ell-7)\cdots(k+\ell-3)
\end{eqnarray}
of length $2(k+\ell-4)$ whose schematic structure is presented in Figure~\ref{structure-of-a-quadrocrucial-perm}.  

\section{Directions of further research}\label{open-sec}

We end our paper by stating a conjecture and a number of open problems (in no particular order). 

\begin{con} Every $(k,\ell)$-crucial permutation can be obtained by deleting elements from a maximal $(k,\ell)$-crucial permutation. 
\end{con}

\begin{prob} \normalfont
Let us fix a Young diagram $\mathcal{D}$. How many crucial permutations of a given type introduced in this paper exist with the shapes of $P$ and $Q$ being $\mathcal{D}$?  The answer to this question is likely to be difficult to find for an arbitrary $\mathcal{D}$, so the problem can be reformulated for a fixed class of Young diagrams, e.g.\ of some regular type, or having other specified properties. 
\end{prob}

For the next problem, recall the definition of an irreducible bicrucial permutation in Section~\ref{5-types-of-crucial-perms-sec}.

\begin{prob} \normalfont
Can we argue that for $k>\ell$ the only irreducible $(k,\ell)$-bicrucial permutations are of length $k+2\ell-5$ and $2k+\ell-5$? (See Remark~\ref{irr-vs-min} and the proof of Theorem~\ref{min-bicrucial-char-thm}). Or do such permutations exist of other lengths?
\end{prob}

\begin{prob}\label{enum-min-tricrucial} \normalfont Generalize Corollary~\ref{min-k3-tricrucial-count-cor} by enumerating minimal $(k,\ell)$-tricrucial permutations. This problem is equivalent to enumerating SYT  $Q$ of the form in Theorem~\ref{min-tricrucial-char-thm} with $q_{1,k-1}<q_{2,2}$. \end{prob}

\begin{prob} \normalfont Enumerate $(k,\ell)$-(right-top,bi,tri,quadro)crucial permutations. \end{prob}

\begin{prob} \normalfont Explain bijectively the connection between $(k,3)$-crucial permutations of next minimal length mentioned in Remark~\ref{next-min-crucial-count} and the combinatorial objects mentioned in the sequence $A077414$ in \cite{oeis}.\end{prob}

For the next problem, recall the definition of a non-extandable permutation in Section~\ref{5-types-of-crucial-perms-sec}.

\begin{prob} \normalfont Note that any non-extendable permutation of any of the five types introduced in this paper is necessarily quadrocrucial. Do non-extendable $(k,\ell)$-quadrocrucial permutations, apart from the maximal $(k,\ell)$-quadrocrucial permutations, exist? If so, characterize such permutations. \end{prob}

\end{document}